\documentclass[12pt]{article}
\usepackage{a4wide}
\usepackage[french,english]{babel}
\usepackage[utf8]{inputenc}
\usepackage[T1]{fontenc}
\usepackage{amsthm}
\usepackage{amssymb}
\usepackage{amsmath}
\usepackage{stmaryrd}
\usepackage{makeidx}
\usepackage{graphicx}
\usepackage{caption}
\theoremstyle{plain}
\usepackage[colorlinks=true,breaklinks=true,linkcolor=black]{hyperref} 
\usepackage{amsfonts}

\newtheorem{thrm}{Theorem}[section] 
\newtheorem{prop}[thrm]{Proposition}

\newtheorem{definition}[thrm]{Definition}
\newtheorem{lem}[thrm]{Lemma}

\newtheorem{remark}[thrm]{Remark}
\date{}
\makeindex
\usepackage{mathrsfs}
\usepackage[all]{xy}
\theoremstyle{definition}

\usepackage{geometry}
\title{A new axiomatics for masures}
\author{Auguste \textsc{Hébert} \\Univ Lyon, UJM-Saint-Etienne CNRS\\ UMR 5208 CNRS, F-42023, SAINT-ETIENNE, France\\ auguste.hebert@univ-st-etienne.fr}

\theoremstyle{definition}\newtheorem{thm*}{Theorem}
\makeatletter \@addtoreset{figure}{section}\makeatother

\newcommand{\R}{\mathbb{R}}
\newcommand{\A}{\mathbb{A}}
\newcommand{\AC}{\mathcal{A}}
\newcommand{\N}{\mathbb{N}}
\newcommand{\Z}{\mathbb{Z}}

\newcommand{\C}{\mathbb{C}}
\newcommand{\Ne}{\mathbb{N}^*}

\newcommand{\I}{\mathcal{I}}
\newcommand{\T}{\mathcal{T}}
\newcommand{\Id}{\mathrm{Id}}
\newcommand{\q}{\mathfrak{q}}
\newcommand{\s}{\mathfrak{s}}
\newcommand{\F}{\mathfrak{F}}

\newcommand{\rr}{\mathfrak{r}}

\newcommand{\RR}{\mathfrak{R}}

\newcommand{\M}{\mathcal{M}}

\newcommand{\Fr}{\mathrm{Fr}}

\newcommand{\In}{\mathrm{Int}}
\newcommand{\conv}{\mathrm{conv}}

\newcommand{\V}{\mathcal{V}}
\newcommand{\HH}{\mathcal{H}}
\newcommand{\supp}{\mathrm{supp}}
\newcommand{\cl}{\mathrm{cl}}

\newcommand{\CC}{\mathcal{C}}
\newcommand{\FF}{\mathcal{F}}

\newcommand{\X}{\mathcal{X}}
\newcommand{\MB}{\mathbb{M}}

\begin{document}

\maketitle

\begin{abstract}
Masures are generalizations of Bruhat-Tits buildings. They were introduced by Gaussent and Rousseau to study Kac-Moody groups over ultrametric fields, which generalize reductive groups. Rousseau gave an axiomatic definition of these spaces. We propose an equivalent axiomatic, which is shorter, more practical and closer to the axiomatic of Bruhat-Tits buildings. Our main tool to prove the equivalence of the axiomatics is the study of the convexity properties in masures.
\end{abstract}

\section{Introduction}
An important tool to study a split reductive group $G$ over a non-archimedean local field is its Bruhat-Tits building defined by Bruhat and Tits in \cite{bruhat1972groupes} and \cite{bruhat1984groupes}. Kac-Moody groups are interesting infinite dimensional (if not reductive) generalizations of reductive groups. In order to study them over fields endowed with a discrete valuation,  Gaussent and Rousseau introduced masures (also known as hovels) in \cite{gaussent2008kac}, which are analogs of Bruhat-Tits buildings. Charignon and Rousseau generalized this construction in \cite{charignon2010immeubles}, \cite{rousseau2017almost}  and \cite{rousseau2016groupes}: Charignon treated the almost split case and Rousseau suppressed restrictions on the base field and on the group. Rousseau also defined an axiomatic of masures in \cite{rousseau2011masures}. Recently, Freyn, Hartnick, Horn and K{\"o}hl made an analog construction in the archimedean case (see \cite{freyn2017kac}): to each split real Kac-Moody group, they associate a space on which the group acts, generalizing the notion of riemannian symmetric space.

Masures enable to obtain results on the arithmetic of (almost)-split Kac-Moody groups over non-archimedean local fields. Let us survey them briefly. Let $G$ be such a group and $\I$ be its masure. In \cite{gaussent2008kac}, Gaussent and Rousseau  use $\I$ to prove a link between the Littlemann's path model and some kind of Mirkovi\'{c}-Vilonen cycle model of $G$. In \cite{gaussent2014spherical}, Gaussent and Rousseau associate a spherical Hecke algebra $^s\mathcal{H}$ to $G$  and they obtain a Satake isomorphism in this setting. These results generalize works of Braverman and Kazhdan obtained when $G$ is supposed affine, see \cite{braverman2011spherical}. In \cite{bardy2016iwahori}, Bardy-Panse, Gaussent and Rousseau define the  Iwahori-Hecke algebra ${^I}\HH$ of $G$.  Braverman, Kazhdan and Patnaik had already done this construction when $G$ is affine in \cite{braverman2016iwahori}. In \cite{hebertGK}, we obtain finiteness results on $G$ enabling to give a meaning to one side of the Gindikin-Karpelevich formula obtained by Braverman, Garland, Kazhdan and Patnaik in the affine case in \cite{braverman2014affine}.  In \cite{abdellatif2017completed}, together with Abdellatif, we define a completion of ${^I}\HH$ and generalize the construction of the Iwahori-Hecke algebra of $G$: we associate Hecke algebras to subgroups of $G$ more general than the Iwahori subgroup, the analogue of the parahoric subgroups. In \cite{bardy2017macdonald}, Bardy-Panse, Gaussent and Rousseau prove a Macdonald's formula for $G$: they give an explicit formula for the image of some basis of $^s\mathcal{H}$ by the Satake isomorphism. Their formula generalizes a well-known formula of Macdonald for reductive groups (see~\cite{macdonald1971spherical}) which had already been extended to affine Kac-Moody groups in \cite{braverman2016iwahori}.

\medskip
 Despite these results some very basic questions are still open in the theory of masures. In this paper we are interested in questions of enclosure maps and of convexity in masures. Let us be more precise. The masure is an object similar to the Bruhat-Tits building. This is a union of subsets called apartments. An apartment is a finite dimensional affine space equipped with a set of hyperplanes called walls. The group $G$ acts by permuting these apartments, which are therefore all isomorphic to one of them called the standard apartment $\A$. 

 To define the masure $\I$ associated to $G$, Gaussent and Rousseau (following Bruhat and Tits) first define $\A$. Let us describe it briefly. Suppose that the field of definition is local. Let  $Q^\vee$ be the coroot lattice of $G$ and $\Phi$ be its set of real roots.  One can consider $Q^\vee$ as a lattice of some affine space $\A$ and $\Phi$ as a set of linear forms on $\A$. Let $Y$ be a lattice of $\A$ containing $Q^\vee$ (one can consider $Y=Q^\vee$ in a first approximation). Then the set $\MB$ of walls of $\A$ is the set of hyperplanes containing an element of $Y$ and whose direction is $\ker(\alpha)$ for some $\alpha\in \Phi$. The half-spaces delimited by walls are called half-apartments. Suppose that $G$ is reductive. Then $\Phi$ is finite and $\I$ is a building. A well known property of buildings is that if $A$ is an apartment of $\I$, then $A\cap \A$ is a finite intersection of half-apartments and there exists an isomorphism from $A$ to $\A$ fixing $A\cap \A$ (see  2.5.7 and Proposition 2.5.8 of \cite{bruhat1972groupes}). Studying this question for masures seems natural for two reasons: first masures generalize Bruhat-Tits buildings and have properties similar to them and second because three of the five axioms of the axiomatic definition of Rousseau are weak forms of this property.

We study this question in the affine case and in the indefinite case.  Let us begin by the affine case, where we prove that this property is true:
 
 \begin{thrm}\label{ThmIntersection convexe affine intro}
Let $\I$ be a masure associated to an affine Kac-Moody group. Let $A$ be an apartment. Then $\A\cap A$ is a finite intersection of half-apartments of $\A$ and there exists an isomorphism from $\A$ to $A$ fixing $\A\cap A$.
\end{thrm}

We define a new axiomatic of masures and prove that it is equivalent to the one given by Rousseau (we recall it in~\ref{subsubDefinition des masures}), using the theorem above. Our axiomatic is simpler and closer to the usual geometric axiomatic of Euclidean buildings. To emphasize this analogy, we first recall one of their definitions in the case where the valuation is discrete   (see Section IV of \cite{brown1989buildings} or Section 6 of \cite{rousseau2004euclidean}, our definition is slightly modified but equivalent). 

\medskip

\begin{definition} A Euclidean building is a set $\I$ equipped with a set $\mathcal{A}$ of subsets called apartments satisfying the following axioms : 

(I0) Each apartment is a Euclidean apartment.

(I1) For any two faces $F$ and $F'$ there exists an apartment containing $F$ and $F'$.

(I2) If $A$ and $A'$ are apartments, then $A\cap A'$ is a finite intersection of half-apartments and there exists an isomorphism $\phi:A\rightarrow A'$ fixing $A\cap A'$.

\end{definition}

In the statement of the next theorem, we use the notion of chimney. They are some kind of thickened sector faces. The word ``splayed'' will be explained later. We prove the following theorem:

\begin{thrm}\label{ThmNouvelle axiomatique affine intro}
Suppose $G$ is an affine Kac-Moody group. Let $\A$ be the apartment associated to the root system of $G$. Let  $(\I,\mathcal{A})$ be a couple such that $\I$ is a set and $\mathcal{A}$ is a set of subsets of $\I$ called apartments. Then $(\I,\mathcal{A})$ is a masure of type $\A$ in the sense of \cite{rousseau2011masures} if and only if it satisfies the following axioms: 

(MA af i) Each apartment is an apartment of type $\A$.

(MA af ii ) If $A$ and $A'$ are two apartments, then $A\cap A'$ is a finite intersection of half-apartments and there exists an isomorphism $\phi:A\rightarrow A'$ fixing $A\cap A'$.

(MA af iii) If $\RR$ is the germ of a splayed chimney and $F$ is a face or a germ of a chimney, then there exists an apartment containing $\RR$ and $F$.

\end{thrm}

We now turn to the general (not necessarily affine) case. Similarly to buildings, we can still define a fundamental chamber $C^v_f$ in the standard apartment $\A$. This enables one to define the Tits cone $\T=\bigcup_{w\in W^v} w.\overline{C^v_f}$, where $W^v$ is the Weyl group of $G$. An important difference between buildings and masures is that when $G$ is reductive, $\T=\A$ and when $G$ is not reductive, $\T\neq \A$ is only a convex cone. This defines a preorder on $\A$ by saying that $x,y\in \A$ satisfy $x\leq y$ if $y\in x+\T$. This preorder extends to a preorder on $\I$ - the Tits preorder - by using isomorphisms of apartments. Convexity properties in $\I$ were known only on preordered pairs of points. If $A,A'$ are apartments and contain two points $x,y$ such that $x\leq y$ then $A\cap A'$ contains the segment in $A$ between $x$ and $y$ and there exists an isomorphism from $A$ to $A'$ fixing this segment (this is Proposition~5.4 of \cite{rousseau2011masures}).

 A ray (half-line) of  $\I$ is said to be generic if its direction meets the interior $\mathring{\T}$ of $\T$. A chimney is splayed if it contains a generic ray. The main result of this paper is the following theorem:

\begin{thrm}\label{ThmIntersection convexe intro}
Let $A$ be an apartment such that $\A\cap A$ contains a generic ray of $\A$. Then $\A\cap A$ is a finite intersection of half-apartments of $\A$ and there exists an isomorphism from $\A$ to $A$ fixing $\A\cap A$.
\end{thrm}

Using this theorem, we prove that the axiomatic definition of Rousseau is equivalent to a  simpler one:

\begin{thrm}\label{ThmNouvelle axiomatique intro}
Let $\A$ be the apartment associated to the root system of $G$. Let  $(\I,\mathcal{A})$ be a couple such that $\I$ is a set and $\mathcal{A}$ is a set of subsets of $\I$ called apartments. Then $(\I,\mathcal{A})$ is a masure of type $\A$ in the sense of \cite{rousseau2011masures} if and only if it satisfies the following axioms: 

(MA i) Each apartment is an apartment of type $\A$.

(MA ii) If two apartments $A$ and $A'$ are such that $A\cap A'$ contains a generic ray, then $A\cap A'$ is a finite intersection of half-apartments and there exists an isomorphism $\phi:A\rightarrow A'$ fixing $A\cap A'$.

(MA iii) If $\RR$ is the germ of a splayed chimney and $F$ is a face or a germ of a chimney, then there exists an apartment containing $\RR$ and $F$.

\end{thrm}  

The axiom (MA iii) (very close to the axiom (MA3) of Rousseau) corresponds to the existence parts of Iwasawa,  Bruhat and  Birkhoff, decompositions  in $G$, respectively for $F$ a face and $\RR$ a sector-germ, $F$ and $\RR$ two sector-germs of the same sign and $F$ and $\RR$ two opposite sector-germs. The axiom (MA ii), which implies the axiom (MA4) of Rousseau, corresponds to the uniqueness part of these decompositions. 

The fact that if $x,y\in \I$ are such that $x\leq y$, the segment between $x$ and $y$ does not depend on the apartment containing $\{x,y\}$ was an axiom of masures (axiom (MAO)). A step of our proof of Theorem~\ref{ThmNouvelle axiomatique intro} is to show that (MAO) is actually a consequence of the other axioms of masures (see Proposition~\ref{propSuppression de MAO}).

\medskip

To define faces and chimneys, Rousseau uses enclosure maps (see~\ref{subsubEnclosure maps} for a precise definition). When $G$ is a reductive group over a local field, the enclosure of a set $P$ of $\A$ is the intersection of the half-apartments of $\A$ containing $P$.  When $G$ is no more reductive $\MB$ can be dense in $\A$. Consequently, Gaussent and Rousseau define the enclosure of a subset to be a filter and no more necessarily a set (which is already the case for buildings when the valuation of the base field is not discrete). Moreover, there are several natural choices of enclosure maps: one can use all the roots (real and imaginary) or only the real roots, one can allow arbitrary intersections of half-apartments or only finite intersections of half-apartments ... This leads to lots of definitions and notations in \cite{rousseau2017almost}. The theorem above proves that all these choices of enclosure maps lead to the same definition of masure; therefore the ``good'' enclosure map is the biggest one, which involves only real roots and finite intersections.

\medskip

Actually we do not limit our study to masures associated to Kac-Moody groups: for us a masure is a set satisfying the axioms of \cite{rousseau2011masures} and whose apartments are associated to a root generating system (and thus to a Kac-Moody matrix). We do not assume  that there exists a group acting strongly transitively on it.  We do not either make any discreteness hypothesis for the standard apartment: if $M$ is a wall, the set of walls parallel to it is not necessarily discrete;  this enables to handle masures associated to split Kac-Moody groups over any ultrametric field.

\medskip

The paper is organized as follows. 

In Section~\ref{secApartment standard}, we describe the general framework and recall the definition of masures.

In Section~\ref{secProprietes generales des intersections d'apparts} we study the intersection of two apartments $A$ and $B$, without assuming that $A\cap B$ contains a generic ray. We prove that $A\cap B$ can be written as a union of enclosed subsets and that $A\cap B$ is enclosed when it is convex. If $P\subset A\cap B$, we give a sufficient condition of existence of an isomorphism from $A$ to $B$ fixing $P$.

In Section~\ref{secIntersection de deux apparts contenant une demi-droite generique}, we study the intersection of two apartments sharing a generic ray and prove Theorem~\ref{ThmIntersection convexe intro}, which is stated as Theorem~\ref{thmIntersection enclose fort}. The reader only interested in masures associated to affine Kac-Moody groups can skip this Section and replace Theorem~\ref{thmIntersection enclose fort} by Lemma~\ref{lemConvexite des intersections d'apparts cas affine ordonne}, which is far more easy to prove.

In Section~\ref{secAxiomatique des masures}, we deduce new axiomatics of masures: we show Theorem~\ref{ThmNouvelle axiomatique intro} and Theorem~\ref{ThmNouvelle axiomatique affine intro}, which correspond to Theorem~\ref{thmAxiomatique des masures} and Theorem~\ref{thmAxiomatique dans le cas affine}.

\subsubsection*{Acknowledgement} I thank St\'ephane Gaussent for discussions on the subject and for his advice on the writing of this paper. I thank Guy Rousseau for discussions on this topic, for his careful reading and comments on a previous version of this paper. 

\subsubsection*{Funding} The author was partially supported by the ANR grant ANR-15-CE40-0012.

\tableofcontents

\section{General framework, Masure}\label{secApartment standard}

In this section, we define our framework and recall the definition of masures. Then we recall some notions on masures. References for this section are \cite{rousseau2011masures}, Section 1 and 2 and Section 1 of \cite{gaussent2014spherical}. 

\subsection{Standard apartment}

\subsubsection{Root generating system}

Let $A$ be a \textbf{Kac-Moody matrix} (also known as generalized Cartan matrix) i.e a square matrix $A=(a_{i,j})_{i,j\in I}$ with integers coefficients, indexed by  a finite set $I$ and satisfying: 
\begin{enumerate}
\item $\forall i\in I,\ a_{i,i}=2$

\item $\forall (i,j)\in I^2|i \neq j,\ a_{i,j}\leq 0$

\item $\forall (i,j)\in I^2,\ a_{i,j}=0 \Leftrightarrow a_{j,i}=0$.
\end{enumerate}

 A  \textbf{root generating system}   of type $A$ is a $5$-tuple $\mathcal{S}=(A,X,Y,(\alpha_i)_{i\in I},(\alpha_i^\vee)_{i\in I})$ made of a Kac-Moody matrix $A$ indexed by $I$, of two dual free $\Z$-modules $X$ (of \textbf{characters}) and $Y$ (of \textbf{cocharacters}) of finite rank $\mathrm{rk}(X)$, a family $(\alpha_i)_{i\in I}$ (of \textbf{simple roots}) in $X$ and a family $(\alpha_i^\vee)_{i\in I}$ (of \textbf{simple coroots}) in $Y$. They have to satisfy the following compatibility condition: $a_{i,j}=\alpha_j(\alpha_i^\vee)$ for all $i,j\in I$. We also suppose that the family $(\alpha_i)_{i\in I}$ is free in $X$ and that the family $(\alpha_i^\vee)_{i\in I}$ is free in $Y$.

Let $\A=Y\otimes \R$. Every element of $X$ induces a linear form on $\A$. We will consider $X$ as a subset of the dual $\A^*$ of $\A$: the $\alpha_i$'s, $i\in I$ are viewed as linear forms on $\A$. For $i\in I$, we define an involution $r_i$ of $\A$ by $r_i(v)=v-\alpha_i(v)\alpha_i^\vee$ for all $v\in \A$. Its space of fixed points is $\ker \alpha_i$. The subgroup of $\mathrm{GL}(\A)$ generated by the $\alpha_i$ for $i\in I$ is denoted by $W^v$ and is called the \textbf{Weyl group} of $\mathcal S$. The system $(W^v,\{r_i|i\in I\})$ is a Coxeter system. For $w\in W^v$, we denote by $\ell(w)$ the length of $w$ with respect to $\{r_i|i\in I\}$.

One defines an action of the group $W^v$ on $\A^*$ by the following way: if $x\in \A$, $w\in W^v$ and $\alpha\in \A^*$ then $(w.\alpha)(x)=\alpha(w^{-1}.x)$. Let $\Phi=\{w.\alpha_i|(w,i)\in W^v\times I\}$, $\Phi$ is the set of \textbf{real roots}. Then $\Phi\subset Q$, where $Q=\bigoplus_{i\in I}\Z\alpha_i$. Let $Q^+=\bigoplus_{i\in I} \N \alpha_i$, $\Phi^+=Q^+\cap \Phi$ and $\Phi^-=(-Q^+)\cap \Phi$. Then $\Phi=\Phi^+\sqcup \Phi^-$. Let $\Delta$ be the set of all roots as defined in \cite{kac1994infinite} and $\Delta_{im}=\Delta\backslash \Phi$. Then $(\A,W^v,(\alpha_i)_{i\in I},(\alpha_i^\vee)_{i\in I},\Delta_{im})$ is a vectorial datum as in  Section 1 of \cite{rousseau2011masures}.

\subsubsection{Vectorial faces and Tits cone}\label{subsubsecVectorial faces}

Define $C_f^v=\{v\in \A|\  \alpha_i(v)>0,\ \forall i\in I\}$. We call it the \textbf{fundamental chamber}. For $J\subset I$, one sets $F^v(J)=\{v\in \A|\ \alpha_i(v)=0\ \forall i\in J,\alpha_i(v)>0\ \forall i\in J\backslash I\}$. Then the closure $\overline{C_f^v}$ of $C_f^v$ is the union of the $F^v(J)$ for $J\subset I$. The \textbf{positive} (resp. \textbf{negative}) \textbf{vectorial faces} are the sets $w.F^v(J)$ (resp. $-w.F^v(J)$) for $w\in W^v$  and $J\subset I$. A \textbf{vectorial face} is either a positive vectorial face or a negative vectorial face. We call \textbf{positive chamber} (resp. \textbf{negative}) every cone  of the form $w.C_f^v$ for some $w\in W^v$ (resp. $-w.C_f^v$).  For all $x\in C_f^v$ and for all $w\in W^v$, $w.x=x$ implies that $w=1$. In particular the action of $w$ on the positive chambers is simply transitive. The \textbf{Tits cone} $\mathcal T$ is defined by $\mathcal{T}=\bigcup_{w\in W^v} w.\overline{C^v_f}$. We also consider the negative cone $-\mathcal{T}$.
We define a $W^v$ invariant preorder $\leq$ (resp. $\mathring{\leq}$) on $\A$, the \textbf{Tits preorder} (the \textbf{Tits open preorder}) by: $\forall (x,y)\in \A\mathrm{}^2$, $x\leq y\ \Leftrightarrow\ y-x\in \mathcal{T}$ (resp. $x\mathring{\leq} y\ \Leftrightarrow\ y-x\in \mathring{\T}\cup\{0\})$.

\subsubsection{Weyl group of  $\A$}\label{subsubGroupe de Weyl}
We now define the Weyl group $W$ of $\A$.   If $X$ is an affine subspace of $\A$, one denotes by $\vec{X}$ its direction. One equips $\A$ with a family $\MB$ of affine hyperplanes called \textbf{real walls} such that:
\begin{enumerate}

\item For all $M\in \MB$, there exists $\alpha_M\in \Phi$ such that $\vec{M}=\ker(\alpha_M)$.

\item For all $\alpha\in \Phi$, there exists an infinite number of hyperplanes $M\in \MB$ such that $\alpha=\alpha_M$.

\item\label{itStabilite des murs sous le groupe de weyl} If $M\in \MB$, we denote by $r_M$ the reflexion of hyperplane $M$ whose associated linear map is $r_{\alpha_M}$. We assume that the group $W$ generated by the $r_M$ for $M\in \MB$ stabilizes $\MB$. 
\end{enumerate}

The group $W$  is the Weyl group of $\A$. A point $x$ is said to be \textbf{special} if every real wall is parallel to a real wall containing $x$. We suppose that $0$ is special and thus $W\supset W^v$. 

If $\alpha\in \A^*$ and $k\in \R$, one sets $M(\alpha,k)=\{v\in \A| \alpha(v)+k=0\}$. Then for all $M\in \MB$, there exists $\alpha\in \Phi$ and $k_M\in \R$ such that $M=M(\alpha,k_M)$. If $\alpha\in \Phi$, one sets $\Lambda_\alpha=\{k_M|\ M\in \mathcal{M}\mathrm{\ and\ }\vec{M}=\ker(\alpha)\}$. Then $\Lambda_{w.\alpha}=\Lambda_\alpha$ for all $w\in W^v$ and $\alpha\in \Phi$. 

If $\alpha\in \Phi$, one denotes by  $\tilde{\Lambda}_\alpha$ the subgroup of $\R$ generated by $\Lambda_\alpha$.  By~(\ref{itStabilite des murs sous le groupe de weyl}), $\Lambda_\alpha=\Lambda_{\alpha}+2\tilde{\Lambda}_\alpha$ for all $\alpha\in \Phi$. In particular, $\Lambda_\alpha=-\Lambda_{\alpha}$ and when $\Lambda_\alpha$ is discrete, $\tilde{\Lambda}_\alpha=\Lambda_\alpha$ is isomorphic to $\Z$. 

One sets $Q^\vee= \bigoplus_{\alpha\in \Phi} \tilde{\Lambda}_\alpha \alpha^\vee$. This is a subgroup of $\A$ stable under the action of $W^v$. Then one has $W=W^v\ltimes Q^\vee$.

 For a first reading, the reader can consider the situation where the walls are the $\phi^{-1}(\{k\})$ for $\phi\in \Phi$ and $k\in \Z$. We then have $\Lambda_\alpha=\Z$ for all $\alpha\in \Phi$, and $Q^\vee=\bigoplus_{i\in I}\Z\alpha_i^\vee$.

\subsubsection{Filters}

\begin{definition}
A filter in a set $E$ is a nonempty set $F$ of nonempty subsets of $E$ such that, for all subsets $S$, $S'$ of $E$,  if $S$, $S'\in F$ then $S\cap S'\in F$ and, if $S'\subset S$, with $S'\in F$ then $S\in F$.
\end{definition}

If $F$ is a filter in a set $E$, and $E'$ is a subset of $E$, one says that $F$ contains $E'$ if every element of $F$ contains $E'$. If $E'$ is nonempty, the set $F_{E'}$ of subsets of $E$ containing $E'$ is a filter. By abuse of language, we will sometimes say that $E'$ is a filter by identifying $F_{E'}$ and $E'$. If $F$ is a filter in $E$, its closure $\overline F$ (resp. its convex envelope) is the filter of subsets of $E$ containing the closure (resp. the convex envelope) of some element of $F$. A filter $F$ is said to be contained in an other filter $F'$: $F\subset F'$ (resp. in a subset $Z$ in $E$: $F\subset Z$) if and only if any set in $F'$ (resp. if $Z$) is in $F$.

If $x\in \A$ and $\Omega$ is a subset of $\A$ containing $x$ in its closure, then the \textbf{germ} of $\Omega$ in $x$ is the filter $germ_x(\Omega)$ of subsets of $\A$ containing a neighborhood of $x$ in $\Omega$.

A \textbf{sector} in $\A$ is a set of the form $\mathfrak{s}=x+C^v$ with $C^v=\pm w.C_f^v$ for some $x\in \A$ and $w\in W^v$. A point $u$ such that $\mathfrak{s}=u+C^v$ is called a \textbf{base point of} $\mathfrak{s}$ and $C^v$ is its \textbf{direction}.  The intersection of two sectors of the same direction is  a sector of the same direction.

The \textbf{sector-germ} of a sector $\mathfrak{s}=x+C^v$ is the filter $\mathfrak{S}$ of subsets of $\A$ containing an $\A$-translate of $\mathfrak{s}$. It only depends on the direction $C^v$. We denote by $+\infty$ (resp. $-\infty$) the sector-germ of $C_f^v$ (resp. of $-C_f^v$).

A ray $\delta$ with base point $x$ and containing $y\neq x$ (or the interval $]x,y]=[x,y]\backslash\{x\}$ or $[x,y]$ or the line containing $x$ and $y$) is called \textbf{preordered} if $x\leq y$ or $y\leq x$ and \textbf{generic} if $y-x\in \pm\mathring \T$, the interior of $\pm \T$.

\subsubsection{Enclosure maps}\label{subsubEnclosure maps}

 Let $\Delta=\Phi\cup\Delta_{im}^+\cup\Delta_{im}^-$ be the set of all roots. For $\alpha\in \Delta$, and $k\in \R\cup\{+\infty\}$, let $D(\alpha,k)=\{v\in \A| \alpha(v)+k\geq 0\}$ (and $D(\alpha,+\infty)=\A\mathrm{}$ for all $\alpha\in \Delta$) and $D^\circ(\alpha,k)=\{v\in \A|\ \alpha(v)+k > 0\}$ (for $\alpha\in \Delta$ and $k\in \R\cup\{+\infty\}$). If $\alpha\in \Delta_{im}$, one sets $\Lambda_\alpha=\R$. Let $[\Phi,\Delta]$ be the set of sets $\mathcal{P}$ satisfying $\Phi\subset \mathcal{P}\subset \Delta$. 

If $X$ is a set, one denotes by $\mathscr{P}(X)$ the set of subsets of $X$. Let $\mathcal{L}$ be the set of families $(\Lambda'_\alpha)\in \mathscr{P}(\R)^\Delta$ such that for all $\alpha\in \Delta$, $\Lambda_\alpha\subset \Lambda'_\alpha$ and $\Lambda_{\alpha}'=-\Lambda_{-\alpha}'$.

Let $\mathscr{F}(\A)$ be the set of filters of $\A$. If $\mathcal{P}\in [\Phi,\Delta]$ and $\Lambda'\in \mathcal{L}$, one defines the map $\cl^\mathcal{P}_{\Lambda'}:\mathscr{F}(\A)\rightarrow \mathscr{F}(\A)$ as follows. If $U\in \mathscr{F}(\A)$, \[\cl^\mathcal{P}_{\Lambda'} (U)=\{V\in U|\ \exists (k_\alpha)\in \prod_{\alpha\in \mathcal{P}}(\Lambda'_\alpha\cup\{+\infty\})|\ V\supset \bigcap_{\alpha\in \mathcal{P}} D(\alpha,k_\alpha)\supset U\}.\] 

If $\Lambda'\in \mathcal{L}$, let $\cl^\#_{\Lambda'}:\mathscr{F}(\A)\rightarrow \mathscr{F}(\A)$  defined as follows. If $U\subset \A$, 
\[\cl^\#_{\Lambda'}(U)=\{V\in U|\ \exists n\in \N, (\beta_i)\in \Phi^n,\ (k_i)\in \prod_{i=1}^n \Lambda'_{\beta_i}|\ V\supset \bigcap_{i=1}^n D(\beta_i,k_i)\supset U\}.\]

Let $\mathcal{CL}^\infty=\{\cl^\mathcal{P}_{\Lambda'}|\mathcal{P}\in [\Phi,\Delta]\mathrm{\ and\ }\Lambda'\in \mathcal{L}\}$. An element of $\mathcal{CL}^\infty$ is called an \textbf{infinite enclosure map}. Let $\mathcal{CL}^\#=\{\cl^\#_{\Lambda'}|\ \Lambda'\in \mathcal{L}\}$. An element of $\mathcal{CL}^\#$ is called  a \textbf{finite} enclosure map. Although $\mathcal{CL}^\infty$ and $\mathcal{CL}^\#$ might not be disjoint (for example if $\A$ is associated to a reductive group over a local field), we define the set of \textbf{enclosure maps} $\mathcal{CL}=\mathcal{CL}^\infty\sqcup \mathcal{CL}^\#$ : in~\ref{subsubDefinition des faces,...}, the definition of the set of faces associated to an enclosure map $\cl$ depends on if $\cl$ is finite or not.

If $\cl\in \mathcal{CL}$, $\cl=\cl^\mathcal{P}_{\Lambda'}$ with $\mathcal{P}\in [\Phi,\Delta]\cup\{\#\}$ and $\Lambda'\in \mathcal{L}$, then for all $\alpha\in \Delta$, $\Lambda'_\alpha=\{k\in \R|\ \cl(D(\alpha,k))=D(\alpha,k)\}$. Therefore  $\cl^\#:=\cl^\#_{\Lambda'}$ is well defined. We do not  use exactly the same notation as Rousseau in \cite{rousseau2017almost} in which $\cl^\#$ means $\cl^\#_\Lambda$.

If $\Lambda'\in \mathcal{L}$, one sets $\mathcal{CL}_{\Lambda'}=\{\cl^\mathcal{P}_{\Lambda'}|\ \mathcal{P}\in [\Phi,\Delta]\}\sqcup \{ \cl^\#_{\Lambda'}\}$.

In order to simplify, the reader can consider the situation where $\Lambda_\alpha=\Lambda'_\alpha=\Z$ for all $\alpha\in \Phi$, $\mathcal{P}=\Delta$ and $\cl=\cl^\Delta_\Lambda$, which is the situation of \cite{gaussent2014spherical}, \cite{bardy2016iwahori} and \cite{hebertGK} for example.  

\medskip

An \textbf{apartment} is a root generating system equipped with a Weyl group $W$ (i.e with a set $\MB$ of real walls, see~\ref{subsubGroupe de Weyl}) and a family $\Lambda'\in \mathcal{L}$. Let $\A=(\mathcal{S},W,\Lambda')$ be an apartment.  A set of the form $M(\alpha,k)$, with $\alpha\in \Phi$ and $k\in \Lambda'_\alpha$  is called a \textbf{wall} of $\A$ and a set of the form $D(\alpha,k)$, with $\alpha\in \Phi$ and $k\in \Lambda'_\alpha$ is called a \textbf{half-apartment} of $\A$. A subset $X$ of $\A$ is said to be enclosed if there exist $k\in \N$, $\beta_1,\ldots,\beta_k\in \Phi$ and $(\lambda_1,\ldots,\lambda_k)\in\prod_{i=1}^k \Lambda'_{\beta_i}$ such that $X=\bigcap_{i=1}^k D(\beta_i,\lambda_i)$ (i.e $X=\cl^\#_{\Lambda'}(X)$). 
 As we shall see, if $\Lambda'\in \mathcal{L}$ is fixed, the definition of masures does not depend on the choice of an enclosure map in $\mathcal{CL}_{\Lambda'}$ and thus it will be more convenient to choose $\cl^\#_{\Lambda'}$, see Theorem~\ref{thmAxiomatique des masures} and Theorem~\ref{thmAxiomatique des masures forte}.

\begin{remark}
Here and in the following, we may replace $\Delta_{im}^+$ by any $W^v$-stable subset of $\bigoplus_{i\in I} \R_+ \alpha_i$ such that $\Delta_{im}^+ \cap \bigcup_{\alpha\in \Phi} \R\alpha$ is empty. We then set $\Delta^-_{im}=-\Delta^+_{im}$. This is useful to include the case of almost split Kac-Moody groups, see 6.11.3 of \cite{rousseau2017almost}.
\end{remark}

\subsection{Masure}

In this section, we define masures. They were introduced in \cite{gaussent2008kac} for symmetrizable split Kac-Moody groups over  ultrametric fields whose residue field contains $\C$, axiomatized in \cite{rousseau2011masures}, then developed and generalized to almost-split Kac-Moody groups over ultrametric fields in \cite{rousseau2016groupes} and \cite{rousseau2017almost}.

\subsubsection{Definitions of  faces, chimneys and related notions}\label{subsubDefinition des faces,...}
Let $\A=(\mathcal{S},W,\Lambda')$ be an apartment. We choose an enclosure map $\cl\in \mathcal{CL}_{\Lambda'}$.

A \textbf{local-face} is associated to a point $x$ and a vectorial face $F^v$ in $\A$; it is the filter $F^\ell(x,F^v)=germ_x(x+F^v)$ intersection of $x+F^v$ and the filter of neighborhoods of $x$ in $\A$.
A \textbf{face} $F$ in $\A$ is a filter associated to a point $x\in \A$ and a vectorial face $F^v\subset \A$. More precisely, if $\cl$ is infinite (resp. $\cl$ is finite), $\cl=\cl^\mathcal{P}_{\Lambda'}$  with $\mathcal{P}\in [\Phi,\Delta]$ (resp. $\cl=\cl^\#_{\Lambda'}$),  $F(x,F^v)$ is the filter made of the subsets containing an intersection  (resp. a finite intersection) of half-spaces $D(\alpha,\lambda_\alpha)$ or $D^\circ(\alpha,\lambda_\alpha)$, with $\lambda
_\alpha\in \Lambda'_\alpha\cup\{+\infty\}$ for all $\alpha\in \mathcal{P}$  (at most one $\lambda_\alpha\in \Lambda_\alpha$ for each $\alpha\in \mathcal{P}$)  (resp. $\Phi$).

There is an order on the faces: if $F\subset \overline{F'}$ one says that``$F$ is a face of $F'$'' or ``$F'$ contains $F$''. The dimension of a face $F$ is the smallest dimension of an affine space generated by some $S\in F$. Such an affine space is unique and is called its \textbf{support}. A face is said to be \textbf{spherical} if the direction of its support meets the open Tits cone $\mathring \T$ or its opposite $-\mathring \T$; then its pointwise stabilizer $W_F$ in $W^v$ is finite.

A \textbf{chamber} (or alcove) is a face of the form $F(x,C^v)$ where $x\in \A$ and $C^v$ is a vectorial chamber of $\A$.

A \textbf{panel} is a face of the form $F(x,F^v)$, where $x\in \A$ and $F^v$ is a vectorial face of $\A$ spanning a wall.

A \textbf{chimney} in $\A$ is associated to a face $F=F(x,F_0^v)$ and to a vectorial face $F^v$; it is the filter $\mathfrak{r}(F,F^v)=\mathrm{cl}(F+F^v)$. The face $F$ is the basis of the chimney and the vectorial face $F^v$ is its direction. A chimney is \textbf{splayed} if $F^v$ is spherical, and is \textbf{solid} if its support (as a filter, i.e., the smallest affine subspace of $\A$  containing $\mathfrak{r}$) has a finite pointwise stabilizer in $W^v$. A splayed chimney is therefore solid. 

A \textbf{shortening} of a chimney $\mathfrak{r}(F,F^v)$, with $F=F(x,F_0^v)$ is a chimney of the form $\mathfrak{r} (F(x+\xi,F_0^v),F^v)$ for some $\xi\in \overline{F^v}$. The \textbf{germ} of a chimney $\mathfrak{r}$ is the filter of subsets of $\A$ containing a shortening of $\mathfrak{r}$ (this definition of shortening is slightly different from the one of \cite{rousseau2011masures} 1.12 but follows \cite{rousseau2017almost} 3.6) and we obtain the same germs with these two definitions).

\subsubsection{Masure}\label{subsubDefinition des masures}

An \textbf{apartment of type }$\A$ is a set $A$ with a nonempty set $\mathrm{Isom}(\A,A)$ of bijections (called \textbf{Weyl-isomorphisms}) such that if $f_0\in \mathrm{Isom}(\A,A)$ then $f\in \mathrm{Isom}(\A,A)$ if and only if, there exists $w\in W$ satisfying $f=f_0\circ w$. We will say \textbf{isomorphism} instead of Weyl-isomorphism in the sequel. An isomorphism between two apartments $\phi:A\rightarrow A'$ is a bijection such that ($f\in \mathrm{Isom}(\mathbb{A},A)$ if, and only if, $\phi \circ f\in \mathrm{Isom}(\A,A')$). We extend all the notions that are preserved by $W$ to each apartment. Thus sectors, enclosures, faces and chimneys are well defined in any apartment of type $\A$.

\begin{definition}
A masure of type $(\A,\cl)$ is a set $\mathcal{I}$ endowed with a covering $\mathcal{A}$ of subsets called \textbf{apartments} such that: 

(MA1) Any $A\in \mathcal{A}$ admits a structure of apartment of type $\A$.

(MA2, $\cl$) If $F$ is a point, a germ of a preordered interval, a generic ray or a solid chimney in an apartment $A$ and if $A'$ is another apartment containing $F$, then $A\cap A'$ contains the enclosure $\mathrm{cl}_A(F)$ of $F$ and there exists an isomorphism from $A$ onto $A'$ fixing $\mathrm{cl}_A(F)$.

(MA3, $\cl$) If $\mathfrak{R}$ is the germ of a splayed chimney and if $F$ is a face or a germ of a solid chimney, then there exists an apartment containing $\mathfrak{R}$ and $F$.

(MA4, $\cl$) If two apartments $A$, $A'$ contain $\mathfrak{R}$ and $F$ as in (MA3), then there exists an isomorphism from $A$ to $A'$ fixing $\mathrm{cl}_A(\mathfrak{R}\cup F)$.

(MAO) If $x$, $y$ are two points contained in two apartments $A$ and $A'$, and if $x\leq_{A} y$ then the two segments $[x,y]_A$ and $[x,y]_{A'}$ are equal.
\end{definition}

In this definition, one says that an apartment contains a germ of a filter if it contains at least one element of this germ. One says that a map fixes a germ if it fixes at least one element of this germ.

The main example of masure is the masure associated to an almost-split Kac-Moody group over an ultrametric field, see \cite{rousseau2017almost}.

\subsubsection{Example: masure associated to a split Kac-Moody group over an ultrametric field}\label{subsecMasure associee} 
Let $A$ be a Kac-Moody matrix and $\mathcal{S}$ be a root generating system of type $A$.
We consider the group functor $\mathbf{G}$ associated to the root generating system $\mathcal{S}$ in \cite{tits1987uniqueness} and in Chapitre 8 of  \cite{remy2002groupes}. This functor is a functor from the category of rings to the category of groups satisfying axioms (KMG 1) to (KMG 9) of \cite{tits1987uniqueness}. When $R$ is a field, $\mathbf{G}(R)$ is uniquely determined by these axioms by Theorem 1' of \cite{tits1987uniqueness}. This functor contains a toric functor $\mathbf{T}$, from the category of rings to the category of commutative groups (denoted $\mathcal{T}$ in \cite{remy2002groupes}) and two functors $\mathbf{U^+}$ and $\mathbf{U^-}$ from the category of rings to the category of groups. 

 Let $\mathcal{K}$ be a field equipped with a non-trivial valuation $\omega:\mathcal{K}\rightarrow \R\cup\{+\infty\}$, $\mathcal{O}$ its ring of integers and $G=\mathbf{G}(\mathcal{K})$ (and $U^+=\mathbf{U^+}(\mathcal{K})$, ...). For all $\epsilon\in \{-,+\}$, and all $\alpha\in \Phi^\epsilon$, we have an isomorphism $x_\alpha$ from $\mathcal{K}$ to a group $U_\alpha$. For all $k\in \R$, one defines a subgroup $U_{\alpha,k}:=x_\alpha(\{u\in \mathcal{K}|\ \omega(u)\geq k\})$. Let $\I$ be the masure associated to $G$ constructed in \cite{rousseau2016groupes}. Then for all $\alpha\in \Phi$, $\Lambda_\alpha=\Lambda'_\alpha=\omega(\mathcal{K})\backslash\{+\infty\}$ and $\cl=\cl_{\Lambda}^\Delta$. If moreover $\mathcal{K}$ is local, one has (up to renormalization, see Lemma 1.3 of \cite{gaussent2014spherical}) $\Lambda_\alpha=\Z$ for all $\alpha\in \Phi$. 
  Moreover, we have: \begin{itemize}
\item[-] the fixer of $\A$ in $G$ is $H=\mathbf{T}(\mathcal{O})$ (by remark 3.2 of \cite{gaussent2008kac})

\item[-] the fixer of $\{0\}$ in $G$ is $K_s=\mathbf{G}(\mathcal{O})$ (by example 3.14 of \cite{gaussent2008kac}).

\item[-] for all $\alpha\in \Phi$ and $k\in \Z$, the fixer of $D(\alpha,k)$ in $G$ is $H.U_{\alpha,k}$ (by 4.2 7) of \cite{gaussent2008kac})

\item[-] for all $\epsilon\in \{-,+\}$, $H.U^\epsilon$ is the fixer of $\epsilon \infty$ (by 4.2 4) of \cite{gaussent2008kac}).

\end{itemize}

If moreover, $\mathcal{K}$ is local, with residue cardinal $q$, each panel is contained in $1+q$ chambers.

The group $G$ is reductive if and only if $W^v$ is finite. In this case, $\I$ is the usual Bruhat-Tits building of $G$ and one has $\T=\A$.

\subsection{Preliminary notions on masures}
In this subsection we recall notions on masures introduced in \cite{gaussent2008kac}, \cite{rousseau2011masures}, \cite{hebertGK} and \cite{hebert2016distances}.

\subsubsection{Tits preorder and Tits open preorder on $\I$}
As the Tits preorder $\leq$ and the Tits open preorder $\mathring{\leq}$ on $\A$ are invariant under the action of $W^v$, one can equip each apartment $A$ with preorders $\leq_A$ and $\mathring{\leq}_A$.  Let $A$ be an apartment of $\I$ and $x,y\in A$ such that $x\leq_A y$ (resp. $x\mathring{\leq}_A y$). Then by Proposition 5.4 of \cite{rousseau2011masures}, if $B$ is an apartment containing $x$ and $y$, then $x\leq_B y$ (resp. $x\mathring{\leq}_B y$). This defines a relation $\leq$ (resp $\mathring{\leq}$) on $\I$.  By Th\'eor\`eme 5.9 of \cite{rousseau2011masures}, this defines a  preorder $\leq$ (resp. $\mathring{\leq}$) on $\I$. It is invariant by isomorphisms of apartments: if $A,B$ are apartments, $\phi:A\rightarrow B$ is an isomorphism of apartments and $x,y\in A$ are such that $x\leq y$ (resp. $x\mathring{\leq } y$), then $\phi(x)\leq \phi(y)$ (resp. $\phi(x)\mathring{\leq} \phi(y)$). We call it the \textbf{Tits preorder on $\I$} (resp. the \textbf{Tits open preorder on $\I$}).

\subsubsection{Retractions centered at sector-germs}\label{subsecRetractions}
Let  $\s$ be a sector-germ of $\I$ and $A$ be an apartment containing it. Let $x\in \I$. By (MA3), there exists an apartment $A_x$ of $\I$ containing $x$ and $\s$. By (MA4), there exists an isomorphism of apartments $\phi:A_x\rightarrow A$ fixing $\s$. By \cite{rousseau2011masures} 2.6, $\phi(x)$ does not depend on the choices we made and thus we can set $\rho_{A,\mathfrak{s}}(x)=\phi(x)$.

The map $\rho_{A,\s}$ is a retraction from $\I$ onto $A$. It only depends on $\s$ and $A$ and we call it the \textbf{retraction onto $A$ centered at $\s$}. 

If $A$ and $B$ are two apartments, and $\phi:A\rightarrow B$ is an isomorphism of apartments fixing some set $X$, one writes $\phi:A\overset{X}{\rightarrow} B$. If $A$ and $B$ share a sector-germ $\q$, one denotes by $A\overset{A\cap B}{\rightarrow} B$ or by $A\overset{\q}{\rightarrow} B$ the unique isomorphism of apartments from $A$ to $B$ fixing $\q$ (and also $A\cap B$). We denote by $\I\overset{\q}{\rightarrow} A$ the retraction onto $A$ fixing $\q$. 
One denotes by $\rho_{+\infty}$ the retraction $\I\overset{+\infty}{\rightarrow } \A$ 
and by $\rho_{-\infty}$ the retraction $\I\overset{-\infty}{\rightarrow} \A$.

\subsubsection{Parallelism in $\I$ and building at infinity}\label{subsubParallelisme}
Let us explain briefly the notion of parallelism  in $\I$. This is done more completely in \cite{rousseau2011masures} Section 3.

Let us begin with rays. Let $\delta$ and $\delta'$ be two generic rays in $\I$. By (MA3) and \cite{rousseau2011masures} 2.2 3) there exists an apartment $A$ containing  sub-rays of $\delta$ and $\delta'$  and we say that $\delta$ and $\delta'$ are \textbf{parallel}, if these sub-rays are parallel in $A$. Parallelism is an equivalence relation and its equivalence classes are called \textbf{directions}.
Let $S$ be a sector of $\I$ and $A$ be an apartment containing $S$. One fixes the origin of $A$ in a base point of $S$. Let $\nu\in S$ and $\delta=\R_+\nu$. Then $\delta$ is a generic ray in $\I$. By Lemma 3.2 of \cite{hebertGK}, for all $x\in \I$, there exists a unique ray $x+\delta$ of direction $\delta$ and base point $x$. To obtain this ray, one can choose an apartment $A_x$ containing $x$ and a sub-ray $\delta'$ of  $\delta$, which is possible by (MA3) and \cite{rousseau2011masures} 2.2 3), and then we take the translate of $\delta'$ in $A_x$ having $x$ as a base point.

A \textbf{sector-face} $f$ of $\A$, is a set of the form $x+F^v$ for some vectorial face $F^v$ and some $x\in \A$. The germ $\F=germ_\infty(f)$ of this sector-face is the filter containing the elements of the form $q+f$, for some $q\in \overline{F^v}$. The sector-face $f$ is said to be spherical if $F^v\cap \mathring{\T}$ is nonempty. A \textbf{sector-panel} is a sector-face contained in a wall and spanning this one as an affine space. A sector-panel is spherical (see \cite{rousseau2011masures} 1).  We extend these notions to $\I$ thanks to the isomorphisms of apartments. Let us make a summary of the notion of parallelism for sector-faces. This is also more complete in \cite{rousseau2011masures}, 3.3.4)). 

If $f$ and $f'$ are two spherical sector-faces, there exists an apartment $B$ containing their germs $\F$ and $\F'$. One says that $f$ and $f'$ are parallel if there exists a vectorial face $F^v$ of $B$ such that  $\F=germ_\infty(x+F^v)$ and $\F'=germ_\infty(y+F^v)$ for some $x,y\in B$. Parallelism is an equivalence relation. The parallelism class of a sector-face germ $\F$ is denoted $\F^\infty$.  We denote by  $\I^\infty$ the set of directions of spherical faces of $\I$. 

By Proposition 4.7.1) of \cite{rousseau2011masures}, for all $x\in \I$ and all $\F^\infty\in \I^\infty$, there exists a unique sector-face $x+\F^\infty$ of direction $\F^\infty$ and with base point $x$. The existence can be obtained in the same way as for rays.

\subsubsection{Distance between apartments}\label{subDistance entre apparts}
Here we recall the notion of distance between apartments introduced in \cite{hebert2016distances}. It will often enable us to make inductions and to restrict our study to apartments sharing a sector. 
Let $\q$ and $\q'$ be two sector germs of $\I$ of the same sign $\epsilon$. By (MA4), there exists an apartment $B$ containing $\q$ and $\q'$. In $B$, there exists a minimal gallery between $\q$ and $\q'$ and the length of this gallery is called the distance between $\q$ and $\q'$. This does not depend on the choice of $B$. If $A'$ is an apartment of $\I$, the distance $d(A',\q)$ between $A'$ and $\q$ is the minimal possible distance between a sector-germ of $A'$ of sign $\epsilon$ and $\q$. If $A$ and $A'$ are apartments of $\I$ and $\epsilon\in \{-1,1\}$, the distance of sign $\epsilon$ between $A$ and $A'$ is the minimal possible distance between a sector-germ of sign $\epsilon$ of $A$ and a sector-germ of sign $\epsilon$ of $A'$. We denote it $d_\epsilon(A,A')$ or $d(A,A')$ if the sign is fixed. 

Let $\epsilon\in \{-,+\}$. Then $d_\epsilon$ is not a distance on the apartments of $\I$ because if $A$ is an apartment, all apartment $A'$ containing a sector of $A$ of sign $\epsilon$ (and there are many of them by (MA3)) satisfies $d_\epsilon(A,A')=0$.

\subsection{Notation}\label{subNotation}

 Let $X$ be a finite dimensional affine space. Let $C\subset X$ be a convex set and $A'$ be its support. The \textbf{relative interior} (resp. \textbf{relative frontier}) of $C$, denoted $\In_r (C)$ (resp. $\Fr_r(C)$) is the interior (resp. frontier) of $C$ seen as a subset of $A'$.  A set is said to be \textbf{relatively open} if it is open in its support.

If $X$ is an affine space and $U\subset X$, one denotes by $\conv(X)$ the convex hull of $X$. 
If $x,y\in \A$, we denote by $[x,y]$ the segment of $\A$ joining $x$ and $y$. If $A$ is an apartment and $x,y\in A$, we denote by $[x,y]_A$ the segment of $A$ joining $x$ and $y$.

If $X$ is a topological space and $a\in X$, one denotes by $\V_X(a)$ the set of open neighborhoods of $a$.

If $X$ is a subset of $\A$, one denotes by $\mathring{X}$ or by $\In (X)$ (depending on the readability) its interior. One denotes by $\Fr(X)$ the boundary (or frontier) of $X$: $\Fr(X)=\overline{X}\backslash \mathring{X}$.

If $X$ is a topological space,  $x\in X$ and $\Omega$ is a subset of $X$ containing $x$ in its closure, then the \textbf{germ} of $\Omega$ in $x$ is denoted $germ_x(\Omega)$.

We use the same notation as in \cite{rousseau2011masures} for segments and segment-germs in an affine space $X$. For example if $X=\R$ and $a,b\in \overline{\R}=\R\cup\{\pm\infty\}$, $[a,b]=\{x\in \overline{\R}|\ a\leq x\leq b\}$, $[a,b[=\{x\in \overline{\R}|\ a\leq x <b\}$, $[a,b)=germ_{a}([a,b])$ ... 

\section{General properties of the intersection of two apartments}\label{secProprietes generales des intersections d'apparts}

In this section, we study the intersection of two apartments, without assuming that their intersection contains a generic ray. 

In Subsection~\ref{subPreliminaires}, we extend results obtained for masure on which a group acts strongly transitively to our framework. 

In Subsection~\ref{subÉcriture de E comme union des Pi}, we write the intersection of two apartments as a finite union of enclosed subsets. 

In Subsection~\ref{subEnclosite d'une intersection convexe}, we use the results of Subsection~\ref{subÉcriture de E comme union des Pi} to prove that if the intersection of two apartments is convex, then it is enclosed. 

In Subsection~\ref{subExistence d'isomorphismes}, we study the existence of isomorphisms fixing subsets of an intersection of two apartments

\medskip

Let us sketch the proof of Theorem~\ref{ThmIntersection convexe intro}. The most difficult part is to prove that if $A$ and $B$ are apartments sharing a generic ray, then $A\cap B$ is convex. We first reduce our study to the case where $A\cap B$ has nonempty interior. We then parametrize the frontier of $A$ and $B$  by a map $\Fr:U\rightarrow \Fr (A\cap B)$, where $U$ is an open and convex set of $A$. The idea is then to prove that for ``almost'' all choices of $x,y\in U$, some map associated to $\Fr_{x,y}:t\in [0,1]\mapsto \Fr(tx+(1-t)y)$ is convex. An important step in this proof is the fact that $\Fr_{x,y}$ is piecewise affine and this relies on the decomposition of Subsection~\ref{subÉcriture de E comme union des Pi}.  The convexity of $A\cap B$ is obtained by using a density argument. We then conclude thanks to Subsection~\ref{subEnclosite d'une intersection convexe} and Subsection~\ref{subExistence d'isomorphismes}.

\subsection{Preliminaries}\label{subPreliminaires}
In this subsection, we extend some results of \cite{hebertGK} and \cite{hebert2016distances}, obtained for a masure on which a group acts strongly transitively to our framework.

\subsubsection{Splitting of apartments}

The following lemma generalizes Lemma 3.2 of \cite{hebert2016distances} to our frameworks:

\begin{lem}\label{lemIntersection de deux apparts contenant un demi-appart}
Let $A_1$ and $A_2$ be two distinct apartments such that $A_1\cap A_2$ contains a half-apartment. Then $A_1\cap A_2$ is a half-apartment.
\end{lem}

\begin{proof}
One identifies $A_1$ and $\A$. By the proof of Lemma 3.2 of \cite{hebert2016distances}, $D=A_1\cap A_2$ is a half-space of the form $D(\alpha,k)$ for some $\alpha\in \Phi$ and $k\in  \R$ (our terminology is not the same as in \cite{hebert2016distances} in which a half-apartment is a half-space of the form $D(\beta,\ell)$, with $\beta\in \Phi$ and $\ell\in \R$, whereas now, we ask moreover that $\ell\in \Lambda'_\beta$).  Let $F,F'$ be opposed sector-panels of $M(\alpha,k)$. Let $S$ be a sector of $D$ dominating $F$, $\s$ its germ and $\mathfrak{F}'$ be the germ of $F'$. By (MA4), one has $A_1\cap A_2\supset \cl(\mathfrak{F}',\s)$. But $\cl(\mathfrak{F}',\s)\supset\cl(D)\supset D =A_1\cap A_2$ and thus $k\in \Lambda'_\alpha$: $A_1\cap A_2$ is a half-apartment.
\end{proof}

As a consequence, one can use Lemma 3.6 and Proposition 3.7 of \cite{hebert2016distances} in our framework. We thus have the following proposition:

\begin{prop}\label{propDecoupages d'apparts}
Let $A$ be an apartment, $\q$ be a sector-germ of $\I$ such that $\q\nsubseteq A$ and $n=d(\q,A)$. \begin{enumerate}
\item\label{itemDecoupage en deux} One can write $A=D_1\cup D_2$, where $D_1$ and $D_2$ are opposite half-apartments of $A$ such that for all $i\in \{1,2\}$, there exists an apartment $A_i$ containing $D_i$ and such that $d(A_i,\q)=n-1$.

\item\label{itemDecoupage en parties encloses} There exist $k\in \N$, enclosed subsets $P_1,\ldots,P_k$ of $A$ such that for all $i\in \llbracket 1,k\rrbracket$, there exist an apartment $A_i$ containing $\q\cup P_i$ and  an isomorphism $\phi_i:A\overset{P_i}{\rightarrow} A_i$.

\end{enumerate}

\end{prop}

\begin{remark}\label{rqueSur les vrais murs}
The choice of the Weyl group $W$ (and thus of $Q^\vee$) imposes restrictions on the walls that can bound the intersection of two apartments. Let $A$ be an apartment and suppose that $A\cap \A=D(\alpha,k)$ for some $\alpha\in \Phi$ and $k\in \Lambda_\alpha'$. Then $k\in \frac{1}{2}\alpha(Q^\vee)$. Indeed, let $D=A\cap \A$, $D_1$ be the half-apartment of $\A$ opposed to $D$ and $D_2$ be the half-apartment of $A$ opposed to $D_1$.  By Proposition 2.9 2) of \cite{rousseau2011masures}  $B=D_1\cup D_2$ is an apartment of $\I$. Let $f:\A\overset{D}{\rightarrow} A$, $g:A\overset{D_2}{\rightarrow} B$ and $h:B\overset{D_1}{\rightarrow} \A$: these isomorphisms exist because two apartments sharing a half-apartment in particular share a sector, see~\ref{subsecRetractions}.  Let $s:\A\rightarrow \A$ making the following diagram commute: \[\xymatrix{ \A \ar[r]^{f} \ar[d]^{s} & A \ar[d]^{g} \\ \A \ar[r]^{h^{-1}} & B.}\]

 The map $s$ fixes $M(\alpha,k)$. Moreover, if $x\in \mathring{D}$, then $f(x)=x$, thus $g(f(x))\in \mathring{D}_1$ and hence $h^{-1}(g(f(x)))\in \mathring{D}_1$. We deduce $s\neq \Id$. The map $s$ is an isomorphism of apartments and thus $s\in W$. As $s$ fixes $M(\alpha,k)$, the vectorial part $\vec{s}$ of $s$ fixes $M(\alpha,0)$. As $W=W^v\ltimes Q^\vee$, one has $s=t\circ \vec{s}$, where $t$ is a translation of vector $q^\vee$ in $Q^\vee$. If $y\in M(\alpha,k)$, one has $\alpha(s(y))=k=\alpha(q^\vee)-k$  and therefore $k\in \frac{1}{2}\alpha(Q^\vee)$. This could enable to be more precise in Proposition~\ref{propDecoupages d'apparts}.
\end{remark}

\subsubsection{A characterization of the points of $\A$}\label{subsubCaracterisation des points de A}

The aim of this subsubsection is to extend Corollary 4.4 of \cite{hebertGK} to our framework.

\subsubsection*{Vectorial distance on $\I$} We recall the definition of the vectorial distance defined in Section 1.7 of \cite{gaussent2014spherical}. Let $x,y\in \I$ be such that $x\leq y$. Then there exists an apartment $A$ containing $x,y$ and an isomorphism $\phi:A\rightarrow\A$. One has $\phi(y)-\phi(x)\in \T$ and thus there exists $w\in W^v$ such that $\lambda=w.(\phi(y)-\phi(x))\in \overline{C^v_f}$. Then $\lambda$ does not depend on the choices we made, it is called the \textbf{vectorial distance between $x$ and $y$} and denoted $d^v(x,y)$. The vectorial distance is invariant under isomorphism of apartments: if $x,y$ are two points in an apartment $A$ such that $x\leq y$, if $B$ is an apartment and if $\phi:A\rightarrow B$ is an isomorphism of apartments, then $d^v(x,y)=d^v(\phi(x),\phi(y))$.

\subsubsection*{Image of a preordered segment by a retraction}

 In Theorem 6.2 of \cite{gaussent2008kac}, Gaussent and Rousseau give a very precise description of the image of a preordered segment by a retraction centered at a sector-germ. However they suppose that a group acts strongly transitively on $\I$. Without this assumption, they prove a simpler property of these images. We recall it here.

 Let $\lambda\in \overline{C^v_f}$. A \textbf{$\lambda$-path} $\pi$ in $\A$ is a map $\pi:[0,1]\rightarrow \A$ such that there exists $n\in \N$ and $0\leq t_1<\ldots< t_n\leq 1$ such that for all $i\in \llbracket 1,n-1\rrbracket$, $\pi$ is affine on $[t_i,t_{i+1}]$ and $\pi'(t)\in W^v.\lambda$ for all $t\in ]t_i,t_{i+1}[$.

\begin{lem}\label{lemImage d'un segment par une retraction}
 Let $A$ be an apartment of $\I$. Let $x,y\in A$ be such that $x\leq y$ and $\rho:\I\rightarrow \A$ be a retraction of $\I$ onto $\A$ centered at a sector-germ $\q$ of $\A$. Let $\tau:[0,1]\rightarrow A$ defined by $\tau(t)=(1-t)x+ty$ for all $t\in [0,1]$ and $\lambda=d^v(x,y)$. Then $\rho\circ \tau$ is a $\lambda$-path between $\rho(x)$ and $\rho(y)$.
 \end{lem}
 
 \begin{proof}
 We rewrite the proof of the beginning of Section 6 of \cite{gaussent2008kac}. Let $\phi:A\rightarrow \A$ be an isomorphism such that $\phi(y)-\phi(x)=\lambda$, which exists by definition of $d^v$. By the same reasoning as in the paragraph of \cite{gaussent2008kac} before Remark 4.6, there exist $n\in \N$, apartments $A_1,\ldots,A_n$ of $\I$ containing $\q$, $0 = t_1 <\ldots < t_n =1$  such that $\tau([t_i,t_{i+1}])\subset A_i$ for all $i\in \llbracket 1,n-1\rrbracket$.
 
Using Proposition 5.4 of \cite{rousseau2011masures}, for all $i\in \llbracket 1,n-1\rrbracket$, one chooses an isomorphism $\psi_i:A\overset{\tau([t_i,t_{i+1}])}{\rightarrow} A_i$. Let $\phi_i:A_i\overset{A_i\cap \A}{\rightarrow} \A$. For all $t\in [t_i,t_{i+1}]$, \[\rho(\tau(t))=\phi_i\circ \psi_i(\tau(t)).\] Moreover, $\phi_i\circ \psi_i:A\rightarrow \A$ and by (MA1), there exists $w_i\in W$ such that $\phi_i\circ \psi_i=w_i\circ \phi$. Therefore for all $t\in ]t_i,t_{i+1}[$, one has $(\rho\circ \tau)'(t)=w_i.\lambda$, which proves that $\rho\circ \tau$ is a $\lambda$-path.
 \end{proof}

\subsubsection*{The projection $y_\nu$} Let $\nu\in C^v_f$ and $\delta=\R^+\nu$. By  paragraph ``Definition of $y_\nu$ and $T_\nu$'' of \cite{hebertGK}, for all $x\in \I$, there exists $y_\nu(x)\in \A$ such that $x+\overline{\delta}\cap \A=y_\nu(x)+\overline{\delta}$, where $x+\overline{\delta}$ is the closure of $x+\delta$ (defined in~\ref{subsubParallelisme}) in any apartment containing it.

\subsubsection*{The $Q^\vee_\R$-order in $\A$} One sets $Q^\vee_{\R,+}=\sum_{\alpha\in \Phi^+}\R_+\alpha^\vee=\bigoplus_{i\in I}\R_+\alpha_i$. One has $Q^\vee_{\R,+}\subset \bigoplus_{i\in I}\R_+ \alpha_i^\vee$. If $x,y\in \A$, one denotes $x\leq _{Q^\vee} y$ if $y-x\in Q^\vee_{\R,+}$.

The following lemma is the writing of Proposition 3.12 d) of \cite{kac1994infinite} in our context. 

\begin{lem}\label{lem2.4 a de GR14}
Let $\lambda\in \overline{C^v_f}$ and $w\in W^v$. Then $w.\lambda\leq_{Q^\vee} \lambda$.
\end{lem}

If $x\in \A$ and $\lambda\in \overline{C^v_f}$, one defines $\pi_\lambda^a: [0,1]\rightarrow \A$ by $\pi^a_\lambda(t)=a+t\lambda$ for all $t\in [0,1]$.

\begin{lem}\label{lemLambda chemin de x à x+ lambda}
Let $\lambda\in \overline{C^v_f}$ and $a\in \A$. Then the unique $\lambda$-path from $a$ to $a+\lambda$ is $\pi_\lambda^a$.
\end{lem}

\begin{proof}
Let $\pi$ be a $\lambda$-path from $a$ to $a+\lambda$. One chooses a subdivision $0 = t_1<\ldots< t_n = 1$ of $[0,1]$ such that for all $i\in \llbracket 1,n-1\rrbracket$, there exists $w_i\in W^v$ such that $\pi_{|[t_i,t_{i+1}]}'(t)= w_i.\lambda$. By Lemma~\ref{lem2.4 a de GR14}, $w_i.\lambda \leq_{Q^\vee} \lambda$ for all $i\in \llbracket 1,n-1\rrbracket$.  Let $h:\bigoplus_{i\in I} \R\alpha_i^\vee\rightarrow \R$ defined by $h(\sum_{i\in I} u_i \alpha_i^\vee)=\sum_{i\in I} u_i$ for all $(u_i)\in \R^I$. Suppose that there exists $i\in \llbracket 1,n-1\rrbracket $ such that $w_i.\lambda\neq \lambda$. Then $h(w_i.\lambda -\lambda)<0$ and for all $j\in \llbracket 1,n-1\rrbracket$, $h(w_j.\lambda-\lambda)\leq 0$.  By integrating, we get that  $h(0)<0$: a contradiction. Therefore $\pi(t)=a+t\lambda=\pi^a_\lambda(t)$ for all $t\in [0,1]$, which is our assertion.
\end{proof}

The following proposition corresponds to Corollary 4.4 of \cite{hebertGK}.

\begin{prop}\label{propCaracterisation des points de l'appartement standard}
Let $x\in \I$ be such that $\rho_{+\infty}(x)=\rho_{-\infty}(x)$. Then $x\in \A$.
\end{prop}

\begin{proof}
Let $x\in \I$ such that $\rho_{+\infty}(x)=\rho_{-\infty}(x)$. Suppose that $x\in \I\backslash \A$. By Lemma 3.5 a) of \cite{hebertGK}, one has $x\leq y_\nu (x)$ and $d^v(x,y_\nu(x))=\lambda$, with $\lambda= y_\nu(x)-\rho_{+\infty}(x)  \in \R^*_+\nu$. Let $A$ be an apartment containing $x$ and $+\infty$, which exists by (MA3). Let $\tau:[0,1]\rightarrow A$ be defined by $\tau(t)=(1-t)x+ty_\nu(x)$ for all $t\in [0,1]$ (this does not depend on the choice of $A$ by Proposition 5.4 of \cite{rousseau2011masures}) and $\pi=\rho_{-\infty}\circ \tau$. Then by Lemma~\ref{lemImage d'un segment par une retraction}, $\pi$ is a $\lambda$-path from $\rho_{-\infty}(x)=\rho_{+\infty}(x)$ to $y_\nu(x)=\rho_{+\infty}(x)+\lambda$.

By Lemma~\ref{lemLambda chemin de x à x+ lambda}, $\pi(t)=\rho_{+\infty}(x)+t\lambda$ for all $t\in [0,1]$.  By Lemma 3.6 of \cite{hebertGK}, $\tau([0,1])\subset \A$. Thus $x=\tau(0)\in \A$: this is absurd. Therefore $x\in \A$, which is our assertion.
\end{proof}

\subsubsection{Topological considerations on apartments}

The following proposition generalizes Corollary 5.9 (ii) of \cite{hebert2016distances} to our framework.

\begin{prop}\label{propRetractions continues}
Let $\q$ be a sector-germ of $\I$ and $A$ be an apartment of $\I$. Let $\rho:\I\overset{\q}{\rightarrow} \A$. Then $\rho_{|A}:A\rightarrow \A$ is continuous (for the affine topologies on $A$ and $\A$).
\end{prop}

\begin{proof}
Using Proposition~\ref{propDecoupages d'apparts}~\ref{itemDecoupage en parties encloses}, one writes $A=\bigcup_{i=1}^n P_i$ where the $P_i$'s are closed subsets of $A$ such that for all $i\in \llbracket 1,n\rrbracket$, there exists an apartment $A_i$ containing $P_i$ and $\q$ and an isomorphism $\psi_i:A\overset{P_i}{\rightarrow} A_i$. For all $i\in \llbracket 1,n\rrbracket$, one denotes by $\phi_i$ the isomorphism $A_i\overset{\q}{\rightarrow} \A$. Then  $\rho_{|P_i}=\phi_i\circ \psi_{i|P_i}$ for all $i\in \llbracket 1,n\rrbracket$. 

Let $(x_k)\in A^\N$ be a converging sequence and $x=\lim x_k$. Then for all $k \in \N$, $\rho(x_k)\in \{\phi_i\circ \psi_i (x_k)|\ i\in \llbracket 1,n\rrbracket\}$ and thus $(\rho(x_n))$ is bounded. Let $(x_{\sigma(k)})$ be a subsequence of $(x_k)$ such that $(\rho(x_{\sigma(k)})$ converges. Maybe extracting a subsequence of $(x_{\sigma(k)})$, one can suppose that there exists $i\in \llbracket 1,n\rrbracket$ such  $x_{\sigma(k)}\in P_i$ for all $k\in \N$. One has $\big(\rho(x_{\sigma(k)})\big)=(\phi_i\circ \psi_i (x_{\sigma(k)}))$ and thus $\rho(x_{\sigma(k)})\rightarrow \phi_i\circ \psi_i(x)=\rho(x)$ (because $P_i$ is closed) and thus $(\rho(x_k))$ converges towards $\rho(x)$. Hence $\rho_{|A}$ is continuous.
\end{proof}

The following proposition generalizes Corollary 5.10 of \cite{hebert2016distances} to our context.

\begin{prop}\label{propIntersection fermee}
Let $A$ be an apartment. Then $A\cap \A$ is closed.
\end{prop}

\begin{proof}
By Proposition~\ref{propCaracterisation des points de l'appartement standard}, $A\cap \A=\{x\in A|\ \rho_{+\infty}(x)=\rho_{-\infty}(x)\}$, which is closed by Proposition~\ref{propRetractions continues}.
\end{proof}

\subsection{Decomposition of the intersection of two apartments into enclosed subsets}\label{subÉcriture de E comme union des Pi}
 
The aim of this subsection is to show that  $\A\cap A$ is a finite union of enclosed subsets of $\A$. 

 We first suppose that $A$ and $\A$ share a sector. One can suppose that $+\infty \subset A\cap \A$. By Proposition~\ref{propDecoupages d'apparts}, one has $A=\bigcup_{i=1}^k P_i$, for some $k\in \N$, where the $P_i$'s are enclosed and $P_i,-\infty$  is contained in some apartment $A_i$ for all $i\in \llbracket 1,k\rrbracket$.

\begin{lem}\label{lemUnion des Pi dont l'interieur est non vide}
Let $X$ be a finite dimensional affine space. Let $U\subset X$ be a set such that $U\subset \overline{\mathring{U}}$ and suppose that $U=\bigcup_{i=1}^n U_i$, where for all $i\in \llbracket 1,n\rrbracket$ the set $U_i$ is the intersection of $U$ and of a finite number of half-spaces. Let $J=\{j\in \llbracket 1,n\rrbracket | \mathring{U_j}\neq  \emptyset \}$. Then $U=\bigcup_{j\in J} U_j$.
\end{lem}

\begin{proof}
Let $j\in \llbracket 1,n\rrbracket$. Then $\Fr (U_j)\cap \mathring{U}$ is contained in a finite number of hyperplanes.  Therefore, if one chooses a Lebesgue measure on $X$, the set $\bigcup_{i\in \llbracket 1,n\rrbracket}\mathring{U}\cap \Fr (U_i)$ has measure $0$ and thus $\mathring{U}\backslash \bigcup_{i\in \llbracket 1,n\rrbracket}\Fr (U_i)$ is dense in $\mathring{U}$ and thus in $U$. Let $x\in U$. Let \[(x_k)\in (\mathring{U}\backslash \bigcup_{i\in \llbracket 1,n\rrbracket}\Fr (U_i))^\N\] be such that $(x_k)$ converges towards $x$. Extracting a sequence if necessary, one can suppose that there exists $i\in \llbracket 1,n\rrbracket$ such that $x_k\in U_i$ for all $k\in \N$. By definition of the frontier, $x_k\in \mathring{U_i}$ for all $k\in \N$. As $U_i$ is closed in $U$, $x\in U_i$ and the lemma follows.
\end{proof}

\begin{lem}\label{lemCondition d'inclusion des Pi}
Let $i\in \llbracket 1, k\rrbracket$ be such that $A\cap \A\cap P_i$ has nonempty interior in $\A$. Then $A\cap \A\supset P_i$.
\end{lem}

\begin{proof}
One chooses an apartment $A_i$ containing $P_i,-\infty$ and $\phi_i:A\overset{P_i}{\rightarrow} A_i$ . Let $\psi_i:A_i\overset{A_i\cap \A}\rightarrow \A$ ($\psi_i$ exists and is unique by Subsection~\ref{subsecRetractions}). Let $x\in P_i$. By definition of $\rho_{-\infty}$, one has $\rho_{-\infty}(x)=\psi_i(x)$ and thus $\rho_{-\infty}(x)=\psi_i\circ \phi_i(x)$. 

Let $f:A\overset{A\cap\A}{\rightarrow} \A$. One has $\rho_{+\infty}(x)=f(x)$ for all $x\in A$. By Proposition~\ref{propCaracterisation des points de l'appartement standard}, \[A\cap \A\cap P_i=\{x\in P_i|\rho_{+\infty}(x)=\rho_{-\infty}(x)\}=P_i\cap (f-\psi_i\circ \phi_i)^{-1}(\{0\}).\]

As $f-\psi_i\circ \phi_i$ is affine, $(f-\psi_i\circ \phi_i)^{-1}(\{0\})$ is an affine subspace of $A$ and as it has nonempty interior, $(f-\psi_i\circ \phi_i)^{-1}(\{0\})=A$. Therefore $P_i\subset \A\cap A$.

\end{proof}

We recall the definition of $x+\infty$, if $x\in \I$ (see~\ref{subsubParallelisme}). Let $x\in \I$ and $B$ be an apartment containing $x$ and $+\infty$. Let $S$ be a sector of $\A$, parallel to $C^v_f$ and such that $S\subset A\cap \A$. Then $x+\infty$ is the sector of $A$ based at $x$ and parallel to $S$. This does not depend on the choice of $A$.

\begin{lem}\label{lemIntersection adherence de son int quartier}
One has $A\cap \A=\overline{\In(A\cap \A)}$.
\end{lem}

\begin{proof}
By Proposition~\ref{propIntersection fermee}, $A\cap \A$ is closed and thus $\overline{\In(A\cap \A)}\subset A\cap \A$.

Let $x\in A\cap \A$. By (MA4), one has $x+\infty \subset A\cap \A$. The fact that there exists $(x_n)\in \In(x+\infty)^\N$ such that $x_n\rightarrow x$ proves the lemma.
\end{proof}

\begin{lem}\label{lemÉcriture de E comme reunion de Pi}
Let  $J=\{i\in \llbracket 1,k\rrbracket|\ \In_\A(P_i\cap A\cap \A)\neq \emptyset\}$. Then $A\cap \A=\bigcup_{j\in J} P_j$. 
\end{lem}

\begin{proof}
Let $U=A\cap\A$. Then by Lemma~\ref{lemIntersection adherence de son int quartier} and Lemma~\ref{lemUnion des Pi dont l'interieur est non vide}, $U=\bigcup_{j\in J} U\cap P_j$ and Lemma~\ref{lemCondition d'inclusion des Pi} completes the proof.
\end{proof}

We no more suppose that $A$ contains $+\infty$. We say that $\bigcup_{i=1}^k P_i$ is a  \textbf{decomposition of $A\cap \A$ into enclosed subsets} if:

 \begin{enumerate}
\item $k\in \N$ and for all $i\in \llbracket 1,k\rrbracket$, $P_i$ is enclosed

\item $A\cap \A=\bigcup_{i=1}^k P_i$

\item for all $i\in \llbracket 1,k\rrbracket$, there exists an isomorphism $\phi_i:\A\overset{P_i}{\rightarrow} A$.

\end{enumerate}

\begin{prop}\label{propÉcriture de l'intersection comme union des Pi, cas general}
Let $A$ be an apartment. Then there exists a decomposition $\bigcup_{i=1}^k P_i$ of $A\cap \A$ into enclosed subsets.

 As a consequence there exists a finite set $\mathcal{M}$ of walls such that $\Fr(A\cap \A)\subset \bigcup_{M\in \M} M$.

If moreover $A\cap \A$ is convex, one has $A\cap \A=\bigcup_{j\in J} P_j$, where $J=\{j\in \llbracket 1,k\rrbracket | \ \supp(P_j)=\supp(A\cap\A)\}$.
\end{prop} 

\begin{proof}
Let $n\in \N$ and $\mathcal{P}_n$: ``for all apartment $B$ such that $d(B,\A)\leq n$, there exists a decomposition  $\bigcup_{i=1}^\ell Q_i$ of $\A\cap B$ into enclosed subsets''. The property $\mathcal{P}_0$ is true by Lemma~\ref{lemÉcriture de E comme reunion de Pi}. Let $n\in \N$ and suppose that $\mathcal{P}_n$ is true. Suppose that there exists an apartment $B$ such that $d(B,\A)=n+1$. Using Proposition~\ref{propDecoupages d'apparts}, one writes $B=D_1\cup D_2$ where $D_1$, $D_2$ are opposite half-apartments such that for all $i\in \{1,2\}$, $D_i$ is contained in a apartment $B_i$ satisfying $d(B_i,\A)=n$. If $i\in \{1,2\}$, one writes $B_i\cap \A=\bigcup_{j=1}^{\ell_i} Q_j^i$, where $\ell_i\in \N$, the $Q_j^i$'s are enclosed and there exists an isomorphism $\psi_j^i:B_i\overset{Q_j^i}{\rightarrow}\A$. Then \[B\cap \A=\bigcup_{j=1}^{\ell_1} (D_1\cap Q_j^1) \cup \bigcup_{j=1}^{\ell_2} (D_2\cap Q_j^2).\] If $i\in \{1,2\}$, one denotes by $f^i$ the isomorphism $B\overset{D_i}{\rightarrow} B_i$. Then if $j\in \llbracket 1, \ell_i\rrbracket$, the isomorphism $\psi_j^i\circ f^i$ fixes $Q_j^i\cap D_i$ and thus $\mathcal{P}_{n+1}$ is true.

Therefore $A\cap \A=\bigcup_{i=1}^k P_i$ where the $P_i$'s are enclosed. One has $\Fr(A\cap \A)\subset \bigcup_{i=1}^k \Fr (P_i)$, which is contained in a finite union of walls.

Suppose that $A\cap \A$ is convex. Let  $X=\supp(A\cap \A)$. By Lemma~\ref{lemUnion des Pi dont l'interieur est non vide} applied with $U=A\cap \A$, \[A\cap \A=\bigcup_{i\in \llbracket1,k\rrbracket,\ \In_X(P_i)\neq \emptyset }P_i,\]

which completes the proof.
\end{proof}

\subsection{Encloseness of a convex intersection}\label{subEnclosite d'une intersection convexe}
In this subsection, we prove Proposition~\ref{propIntersection enclose quand convexe}: if $A$ is an apartment such that $A\cap \A$ is convex, then $A\cap \A$ is enclosed. For this we study the ``gauge'' of $A\cap \A$, which is a map parameterizing the frontier of $A\cap \A$.

\begin{lem}\label{lemSupport d'une partie enclose}
Let $A$ be a finite dimensional affine space, $k\in \Ne$ and  $D_1,\ldots,D_k$ be half-spaces of $A$ and $M_1,\ldots, M_k$ be their hyperplanes. Then their exists $J\subset \llbracket 1,k\rrbracket$ (maybe empty) such that $\mathrm{supp}(\bigcap_{i=1}^k D_i) =\bigcap_{j\in J} M_j$
\end{lem}

\begin{proof}
Let $d\in \N^*$ and $\ell\in \N$. Let $\mathcal{P}_{d,\ell}$:``for all affine space $X$ such that $\dim X\leq d$ and for all half-spaces $E_1,\ldots, E_\ell$ of $X$, there exists $J\subset \llbracket 1,\ell \rrbracket $ such that $\mathrm{supp}( \bigcap_{i=1}^\ell E_i)=\bigcap_{j\in J} H_j$ where for all $j\in J$, $H_j$ is the hyperplane of $E_j$''.

It is clear that for all $\ell\in \N$, $\mathcal{P}_{1,\ell}$ is true and that for all $d\in \N$, $\mathcal{P}_{d,0}$ and $\mathcal{P}_{d,1}$ is true. Let $d\in \N_{\geq 2}$ and $\ell\in \N$ and suppose that (for all $d'\leq d-1$ and $\ell'\in \N$, $\mathcal{P}_{d',\ell'}$ is true) and that (for all $\ell'\in \llbracket 0,\ell\rrbracket$, $\mathcal{P}_{d,\ell'}$ is true). 

Let $X$ be a $d$ dimensional affine space, $E_1,\ldots,E_{\ell+1}$ be half-spaces of $X$ and $H_1,\ldots, H_{\ell+1}$ be their hyperplanes. Let $L=\bigcap_{j=1}^{\ell} E_j$ and $S=\mathrm{supp}\ L$. Then $E_{\ell+1}\cap S$ is either $S$ or a half-space of $S$. In the first case, $E_{\ell+1}\supset S\supset L$, thus $\bigcap_{i=1}^{\ell+1} E_i=L$ and thus by $\mathcal{P}_{d,\ell}$, $\mathrm{supp}(\bigcap_{i=1}^{\ell+1} E_i)=\bigcap_{j\in J} H_j$  for some $J\subset \llbracket 1,\ell\rrbracket$.

Suppose that $E_{\ell+1}\cap S$ is a half-space of $S$. Then either $\mathring{E}_{\ell+1}\cap L\neq \emptyset$ or $\mathring{E}_{\ell+1}\cap L=\emptyset$. In the first case, one chooses $x\in \mathring{E}_{\ell+1}\cap L$ and  a sequence $(x_n)\in (\In_r (L))^\N$ converging towards $x$. Then for $n\gg 0$, $x_n\in  \mathring{E}_{\ell+1}\cap \In_r(L)$. Consequently,  $L\cap E_{\ell+1}$ has nonempty interior in $S$. Thus $\mathrm{supp}(\bigcap_{i=1}^{\ell+1} E_i)=S$ and  by $\mathcal{P}_{d,\ell}$, $\mathrm{supp}(\bigcap_{i=1}^{\ell+1} E_i)=\bigcap_{j\in J} H_j$  for some $J\subset \llbracket 1,\ell\rrbracket$.

Suppose now that $\mathring{E}_{\ell+1} \cap L$ is empty. Then $L\cap E_{\ell+1}\subset H_{\ell+1}$, where $H_{\ell+1}$ is the hyperplane of $E_{\ell+1}$. Therefore $\bigcap_{i=1}^{\ell+1} E_i= \bigcap_{i=1}^{\ell+1} (E_i\cap H_{\ell+1})$  and thus by $\mathcal{P}_{d-1,\ell+1}$, $\mathrm{supp}(\bigcap_{i=1}^{\ell+1} E_i)=\bigcap_{j\in J} H_j$  for some $J\subset \llbracket 1,\ell+1\rrbracket$.

\end{proof}

\begin{lem}\label{lemSupport d'une intersection convexe}
Let $A$ be an apartment such that $A\cap \A$ is convex. Then $\supp (A\cap \A)$ is enclosed.
\end{lem}

\begin{proof}
Using Proposition~\ref{propÉcriture de l'intersection comme union des Pi, cas general},  one writes $A\cap \A=\bigcup_{i=1}^k P_i$, where the $P_i$'s are enclosed and $\supp(P_i) =\supp (A\cap \A)$ for all $i\in \llbracket 1,k\rrbracket$. By Lemma~\ref{lemSupport d'une partie enclose}, if $i\in \llbracket 1,k\rrbracket$, then $\supp (P_i)$ is a finite intersection of walls, which proves the lemma.
\end{proof}

\subsubsection*{Gauge of a convex}

Let $A$ be a finite dimensional affine space. Let $C$ be a closed and convex  subset of $A$ with nonempty interior. One chooses $x\in \mathring{C}$ and one fixes the origin of $A$ in $x$. Let $j_{C,x}:A\rightarrow \R_+\cup\{+\infty\}$ defined by \[j_{C,x}(s)=\inf \{t\in\R^*_+|s\in tC\}.\] The map $j_{C,x}$ is called the \textbf{gauge} of $C$ based at $x$. In the sequel, we will fix some $x\in \mathring{C}$ and we will denote $j_C$ instead of $j_{C,x}$. Then by Theorem 1.2.5 of \cite{hiriart2012fundamentals} and discussion at the end of Section 1.2 of loc cit, $j_C(A)\subset \R_+$ and $j_C$ is continuous.

The following lemma is easy to prove: 

\begin{lem}\label{lemDescription d'un convexe par sa jauge}
Let $C$ be a convex closed set with nonempty interior. Fix the origin of $A$ in a point of  $\mathring{C}$. Then $C=\{x\in A| j_C(x)\leq 1\}$ and $\mathring{C}=\{x\in A| j_C(x) <1\}$.
\end{lem}

\begin{lem}\label{lemDescription de la frontiere par jauge}
Let $C$ be a convex closed set with nonempty interior. Fix the origin of $A$ in $\mathring{C}$. Let $U=U_C=\{s\in A|\ j_C(s)\neq 0\}$. Let $\Fr=\Fr_C:U\rightarrow \Fr (C)$ defined by $\Fr(s)=\frac{s}{j_C(s)}$ for all $s\in U$. Then $\Fr$ is well defined, continuous and surjective.
\end{lem}

\begin{proof}
If $s\in U$, then $j_C(\Fr(s))=\frac{j_C(s)}{j_C(s)}=1$ and thus $\Fr$ takes it values in $\Fr (C)$ by Lemma~\ref{lemDescription d'un convexe par sa jauge}. The continuity of $\Fr$ is a consequence of the one of $j_C$. 

Let $f\in \Fr(C)$. Then $\Fr(f)=f$ and thus $\Fr$ is surjective. 

\end{proof}

Let $A$ be an apartment such that $A\cap \A$ is convex and nonempty. Let $X$ be the support of $A\cap \A$ in $\A$. By Lemma~\ref{lemSupport d'une intersection convexe}, if $A\cap \A=X$, then  $A\cap \A$ is enclosed. One now supposes that $A\cap \A\neq X$. One chooses $x_0\in \In_X(A\cap \A)$ and consider it as the origin of $\A$. One defines $U=U_{A\cap \A}$ and $\Fr: U\rightarrow \Fr_r(A\cap \A)$ as in Lemma~\ref{lemDescription de la frontiere par jauge}. The set  $U$ is open and nonempty. Using Proposition~\ref{propÉcriture de l'intersection comme union des Pi, cas general},  one writes $A\cap \A=\bigcup_{i=1}^r P_i$, where $r\in \N$, the $P_i$'s are enclosed and $\supp( P_i) =X$ for all $i\in \llbracket 1,r\rrbracket$.  Let $M_1,\ldots, M_k$ be distinct walls not containing $X$ such that $\Fr_r(A\cap \A)\subset \bigcup_{i=1}^k M_i$, which exists because the $P_i$'s are intersections of half-spaces of $X$ and $A\cap \A\neq X$. Let $\M=\{M_i\cap X|i\in \llbracket 1,k\rrbracket\}$. If $M\in \M$, one sets $U_M=\Fr^{-1}(M)$.

\begin{lem}\label{lemDensite de U'}
Let $U'=\{x\in U|\exists (M,V)\in \M\times \V_U(x) |\Fr(V)\subset M\}$. Then $U'$ is dense in $U$. 
\end{lem}

\begin{proof}
 Let $M\in \M$. By Lemma~\ref{lemDescription de la frontiere par jauge}, $U_M$ is closed in $U$. Let $V'\subset U$ be nonempty and open. Then $V'=\bigcup_{M\in \M} U_M\cap V'$. As $\M$ is finite, we can apply Baire's Theorem, and there exists $M\in \M$ such that $V'\cap U_M$ has nonempty interior and hence $U'$ is dense in $U$.
\end{proof}

\begin{lem}\label{lemIndependance mur associe à un point}
Let $x\in U'$ and $V\in \V_U(x)$ be such that $\Fr(V)\subset M$ for some $M\in \M$. The wall $M$ is unique and does not depend on $V$. 	
\end{lem}

\begin{proof}
Suppose that $\Fr(V)\subset M\cap M'$, where $M,M'$ are hyperplanes of $X$. Let $\alpha,\alpha'\in \Phi$, $k,k'\in \R$ be such that $M=\alpha^{-1}(\{k\})$ and $M'=\alpha'^{-1}(\{k'\})$. By definition of $U$, for all $y\in V$, $\Fr(y)=\lambda(y) y$ for some $\lambda(y)\in \R^*_+$. Suppose that $k=0$. Then $\alpha(y)=0$ for all $y\in V$, which is absurd because $\alpha\neq 0$. By the same reasoning, $k'\neq 0$. 

 If $y\in V\backslash \big(\alpha^{-1}(\{0\})\cup \alpha'^{-1}(\{0\})\big)$, $\Fr(y)=\lambda(y)y$ for some $\lambda(y)\in \R^*_+$ and thus $\Fr(y)=\frac{k}{\alpha(y)}y=\frac{k'}{\alpha'(y)}y$. As $V\backslash \big(\alpha^{-1}(\{0\})\cup \alpha'^{-1}(\{0\})\big)$ is dense in $V$, $k\alpha'(y)=k'\alpha(y)$ for all $y\in V$ and thus $M$ and $M'$ are parallel. Therefore $M=M'$. It remains to show that $M$ does not depend on $V$. Let $V_1 \in \V_U(x)$ be such that $\Fr(V_1)\subset M_1$ for some $M_1\in \M$. By the uniqueness we just proved applied to $V\cap V_1$, one has $M=M_1$, which completes the proof.
\end{proof}

If $x\in U'$, one denotes by $M_x$ the wall defined by Lemma~\ref{lemIndependance mur associe à un point}.

\begin{lem}\label{lemInclusion dans un demi-apart}
Let $x\in U'$ and $D_1$, $D_2$ be the two half-spaces of $X$ defined by $M_x$. Then either $A\cap \A\subset D_1$ or $A\cap \A\subset D_2$.
\end{lem}

\begin{proof}
Let $V\in \V_U(x)$ be such that $\Fr(V)\subset M_x$. Let us prove that $\Fr(V)=\R^*_+V\cap M_x$.  As $\Fr(y)\in \R^*_+ y$ for all $y\in V$, $\Fr(V)\subset \R^*_+ V\cap M_x$.  Let $f$ be a linear form on $X$ such that $M_x=f^{-1}(\{k\})$ for some $k\in \R$. If $k=0$, then $f(v)=0$ for all $v\in V$, and thus $f=0$: this is absurd. Hence $k\neq 0$.  

Let $a\in \R^*_+ V\cap M_x$. One has $a=\lambda \Fr(v)$, for some $\lambda\in \R^*_+$ and $v\in V$. Moreover $f(\Fr(v))=k=f(a)$ and as $k\neq 0$, $a=\Fr(v)\in \Fr(V)$. Thus $\Fr(V)= \R^*_+ V\cap M_x$  and $\Fr(V)$ is an open set of $M_x$. Suppose there exists $(x_1,x_2)\in (\mathring{D_1}\cap A\cap \A)\times (\mathring{D}_2 \cap A\cap \A)$. Then $\conv(x_1,x_2,\Fr(V))\subset A\cap \A$ is an open neighborhood of $\Fr(V)$ in $X$. This is absurd because $\Fr$ takes it values in $\Fr_r(A\cap \A)$. Thus the lemma is proved.
\end{proof}

If $x\in U'$, one denotes by $D_x$ the half-space delimited by $M_x$ and containing $A\cap \A$.

\begin{prop}\label{propIntersection enclose quand convexe}
Let $A$ be an apartment such that $A\cap \A$ is convex. Then $A\cap \A$ is enclosed.
\end{prop}

\begin{proof}

If $u\in U'$, then $A\cap \A\subset D_u$ and thus $A\cap \A\subset \bigcap_{u\in U'} D_u$. 

Let $x\in U'\cap \bigcap_{u\in U'} D_u$. One has $0\in A\cap \A$ and thus $0\in D_x$. Moreover $\Fr(x)\in M_x\cap A\cap \A$ and thus $x\in [0,\Fr(x)]\subset A\cap \A$. Therefore \[U'\cap \bigcap_{x\in U'} D_x\subset A\cap \A.\]

Let $x\in \In_X (\bigcap_{u\in U'}D_u)$. If $x\notin U$, then $x\in A\cap \A$. Suppose $x\in U$. Then by Lemma~\ref{lemDensite de U'}, there exists $(x_n)\in (U'\cap \In_X (\bigcap_{u\in U'}D_u) )^\N$ such that  $x_n\rightarrow x$. But then for all $n\in \N$, $x_n\in A\cap \A$ and by Proposition~\ref{propIntersection fermee}, $x\in A\cap \A$. As a consequence, $A\cap \A \supset \In_X (\bigcap_{u\in U'}D_u)$. As $A\cap \A$ is closed, \[A\cap \A \supset \overline{\In_X (\bigcap_{u\in U'}D_u)}=\bigcap_{u\in U'}D_u\] because $\bigcap_{u\in U'}D_u$ is closed, convex with nonempty interior in $X$. Thus we have proved $A\cap \A=\bigcap_{u\in U'}D_u$. 

Let $M_1',\ldots,M_k'$ be walls of $\A$ such that for all $x\in U'$, there exists $i(x)\in \llbracket 1,k\rrbracket$ such that $M_{i(x)}'\cap X=M_x$. One sets $M'_x=M'_{i(x)}$ for all $x\in U'$ and one denotes by $D_x'$ the half-apartment of $\A$ delimited by $M_x'$ and containing $D_x$. Then $X\cap \bigcap_{x\in U'} D_x'=A\cap \A$. Lemma~\ref{lemSupport d'une intersection convexe} completes the proof. 

\end{proof}

\subsection{Existence of isomorphisms of apartments fixing a convex set}\label{subExistence d'isomorphismes}

Let $A$ be an apartment and $P\subset \A\cap A$. In this section, we study the existence of isomorphisms of apartments $\A\overset{P}{\rightarrow}A$. We give a sufficient condition of existence of such an isomorphism in Proposition~\ref{propIsomorphisms fixant une partie convexe}. The existence of an isomorphism $A\overset{A\cap \A}{\rightarrow} \A$ when $A$ and $\A$ share a generic ray will be a particular case of this Proposition, see Theorem~\ref{thmIntersection enclose fort}. In the affine case, this will be a first step to prove that for all apartment $A$, there exists an isomorphism $A\overset{A\cap \A}{\rightarrow} \A$.

\begin{lem}\label{lemIso fixant le cone engendre}
Let $A$ be an apartment of $\I$ and $\phi:\A\rightarrow A$ be an isomorphism of apartments. Let $P\subset \A\cap A$ be a nonempty relatively open convex set, $Z=\supp(P)$ and suppose that $\phi$ fixes $P$. Then $\phi$ fixes $P+(\T\cap \vec{Z})\cap A$, where $\T$ is the Tits cone.
\end{lem}

\begin{proof}
Let $x\in P+(\T\cap \vec{Z})\cap A$. Write $x=p+t$, where $p\in P$ and $t\in \T$. Asssume $t\neq 0$. Let $L=p+\R t$.  Then $L$ is a preordered line in $\I$ and $\phi$ fixes $L\cap P$. Moreover, $p\leq x$ and thus by Proposition 5.4 of \cite{rousseau2011masures}, there exists an isomorphism $\psi:\A\overset{[p,x]}{\rightarrow} A$. In particular, $\phi^{-1}\circ\psi :\A\rightarrow \A$ fixes $L\cap P$. But then $\phi^{-1}\circ \psi_{|L}$ is an affine isomorphism fixing a nonempty open set of $L$: this is the identity. Therefore $\phi^{-1}\circ \psi(x)=x=\phi^{-1}(x)$, which shows the lemma.
 \end{proof}
 
\begin{lem}\label{lemIsomorphisms fixant un sous ouvert}
Let $A$ be an apartment of $\I$. Let $U\subset \A\cap A$ be a nonempty relatively open set and $X=\supp(U)$. Then there exists a nonempty open subset $V$ of $U$ (in $X$) such that there exists an isomorphism $\phi:\A\overset{V}{\rightarrow} A$.
\end{lem} 

\begin{proof}
Let $\bigcup_{i=1}^k P_i$ be a decomposition into enclosed subsets of $A\cap \A$. Let $i\in \llbracket 1,k\rrbracket$ be such that $P_i\cap U$ has nonempty interior in $X$ and $\phi:\A\overset{P_i}{\rightarrow} A$. Then $\phi$ fixes a nonempty open set of $U$, which proves the lemma.
  \end{proof}
  
  \begin{lem}\label{lemEns des points fixes par un iso ferme}
  Let $A$ be an apartment of $\I$ and $\phi:\A\rightarrow A$ be an isomorphism. Let $F=\{z\in A| \phi(z)=z\}$. Then $F$ is closed in $\A$.
  \end{lem}
  
  \begin{proof}
By Proposition~\ref{propRetractions continues}, $\rho_{+\infty}\circ \phi:\A\rightarrow \A$ and $\rho_{-\infty}\circ \phi:\A\rightarrow \A$ are continuous. Let $(z_n)\in F^\N$ be such that $(z_n)$ converges in $\A$ and $z=\lim z_n$.

For all $n\in \N$, one has \[\rho_{+\infty}(\phi(z_n))=z_n=\rho_{-\infty}(\phi(z_n))\rightarrow \rho_{+\infty}(\phi(z))=z=\rho_{-\infty}(\phi(z)).\] By Proposition~\ref{propCaracterisation des points de l'appartement standard}, $z=\phi(z)$, which proves the lemma.
  \end{proof}
 
 \begin{prop}\label{propIsomorphisms fixant une partie convexe}
 Let $A$ be an apartment of $\I$ and $P\subset \A\cap A$ be a convex set. Let $X=\supp(P)$ and suppose that $\T\cap \vec{X}$ has nonempty interior in $\vec{X}$. Then there exists an isomorphism of apartments $\phi:\A\overset{P}{\rightarrow} A$.
 \end{prop}
 
 \begin{proof} (see Figure~\ref{figIso fixant une partie})
 Let $V\subset P$ be a nonempty open set of $X$ such that there exists an isomorphism $\phi:\A\overset{V}{\rightarrow} A$ (such a $V$ exists by Lemma~\ref{lemIsomorphisms fixant un sous ouvert}). Let us show that $\phi$ fixes $\In_r(P)$.

Let $x\in V$. One fixes the origin of $\A$ in $x$ and thus $X$ is a vector space. Let $(e_j)_{j\in J}$ be a basis of $\A$ such that for some subset $J'\subset J$, $(e_j)_{j\in J'}$ is a basis of $X$ and $(x+\T)\cap X\supset \bigoplus_{j\in J'}\R^*_+ e_j$. For all $y\in X$, $y=\sum_{j\in J'}y_j e_j$ with $y_j\in \R$ for all $j\in J'$, one sets $|y|=\max_{j\in J'}|y_j|$. If $a\in A$ and $r>0$, one sets $B(a,r)=\{y\in X|\ |y-a|<r\}$.
 
   Suppose that $\phi$ does not fix $\In_r(P)$. Let $y\in \In_r(P)$ be such that $\phi(y)\neq y$. Let \[s=\sup\{t\in [0,1]|\exists U\in \V_X([0,ty])|\ \phi\mathrm{\ fixes\ }U\}.\] Set $z=sy$. Then by Lemma~\ref{lemEns des points fixes par un iso ferme}, $\phi(z)=z$.
   
By definition of $z$, for all $r>0$, $\phi$ does not fix $B(z,r)$.    Let $r>0$ be such that $B(z,5r)\subset \In_r P$. Let $z_1\in B(z,r)\cap [0,z[$ and $r_1>0$  be such that $\phi$ fixes $B(z_1,r_1)$ and $z_2'\in B(z,r)$ such that $\phi(z_2')\neq z_2'$. Let $r_2'\in ]0,r[$ be such that for all $a\in B(z_2',r_2')$, $\phi(z)\neq z$. Let $z_2\in B(z_2',r_2')$ be such that for some $r_2\in ]0,r_2'[$, $B(z_2,r_2)\subset B(z_2',r_2')$ and such that there exists an isomorphism $\psi:\A\overset{B(z_2,r_2)}{\rightarrow} A$ (such $z_2$ and $r_2$ exists by Lemma~\ref{lemIsomorphisms fixant un sous ouvert}). Then $|z_1-z_2| < 3r$.

Let us prove that $(z_1+\T\cap X)\cap (z_2+\T\cap X)\cap \In_r( P)$ contains a nonempty open set $U\subset X$. One identifies $X$ and $\R^{J'}$ thanks to the basis $(e_j)_{j\in J'}$. One has $z_2-z_1\in ]-3,3[^{J'}$ and thus \[(z_1+\T)\cap (z_2+\T)=(z_1+\T)\cap (z_1+z_2-z_1+\T)\supset z_1+ ]3,4[^{J'}.\] As $P\supset B(z_1,4r)$, the set $(z_1+\T\cap X)\cap (z_2+\T\cap X)\cap \In_r( P)$ contains a nonempty open set $U\subset X$.
   
   By Lemma~\ref{lemIso fixant le cone engendre}, $\phi$ and $\psi$ fix $U$. Therefore, $\phi^{-1}\circ \psi$ fixes $U$ and as it is an isomorphism of affine space of $A$, $\phi^{-1}\circ \psi$ fixes $X$. Therefore $\phi^{-1}\circ \psi(z_2)=\phi^{-1}(z_2)=z_2$ and thus $\phi(z_2)=z_2$: this is absurd. Hence $\phi$ fixes $\In_r(P)$. By Lemma~\ref{lemEns des points fixes par un iso ferme}, $\phi$ fixes $\overline{\mathrm{\In_r}(P)}=\overline{P}$ and thus $\phi$ fixes $P$, which shows the proposition. 
 \end{proof}

\begin{figure}
 \centering
\includegraphics[scale=0.3]{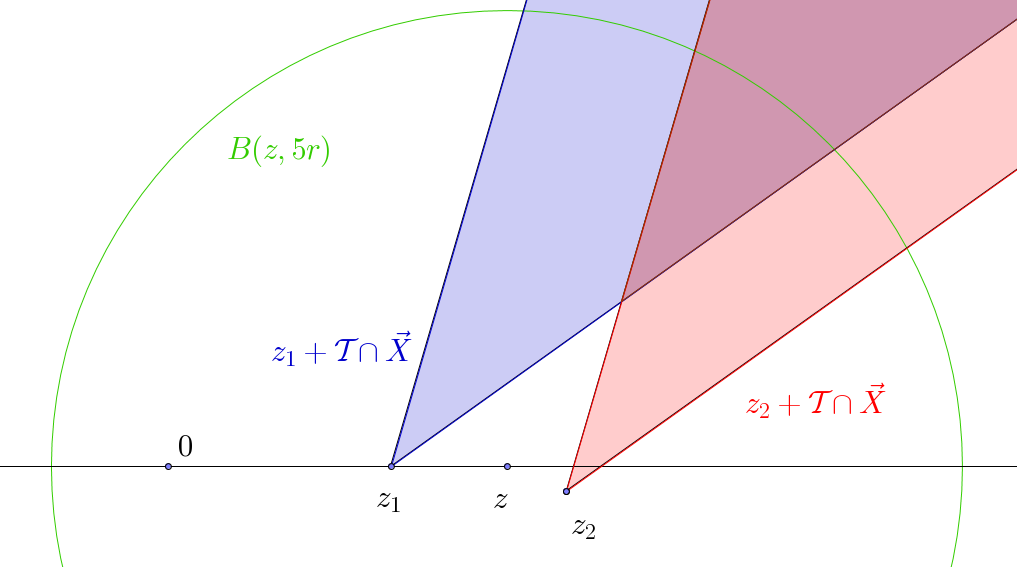}
\caption{Proof of Proposition~\ref{propIsomorphisms fixant une partie convexe}}\label{figIso fixant une partie}
\end{figure}

\section{Intersection of two apartments sharing a generic ray}\label{secIntersection de deux apparts contenant une demi-droite generique}

 The aim of this section is to prove Theorem~\ref{thmIntersection enclose fort}: let $A$ and $B$ be two apartments sharing a generic ray. Then  $A\cap B$ is enclosed and there exists an isomorphism $\phi:A\overset{A\cap B}{\rightarrow} B$.  

We first reduce our study to the case where $A\cap B$ has nonempty interior by the following lemma:

\begin{lem}\label{lemReduction au cas d'interieur non vide}
Suppose that for all apartments $A,B$ such that $A\cap B$ contains a generic ray and has nonempty interior, the set $A \cap B$ is convex. Then if $A_1$ and $A_2$ are two apartments containing a generic ray, the set $A_1\cap A_2$ is enclosed and there exists an isomorphism $\phi:A_1\overset{A_1\cap A_2}{\rightarrow} A_2$.   
\end{lem}

\begin{proof}
 Let us prove that $A_1\cap A_2$ is convex. Let $\delta$ be the direction of a generic ray shared by $A_1$ and $A_2$. Let $x_1,x_2\in A_1\cap A_2$ and $\F^\infty$ be the vectorial face direction containing $\delta$. Let $\F'^\infty$ be the vectorial face direction of $A_1$ opposite to $\F^\infty$. Let $C_1$ be a chamber of $A_1$ containing $x_1$ and  $C_2$ be a chamber of $A_2$ containing  $x_2$. Set $\rr_1=\rr(C_1,\F'^\infty)\subset A_1$, $\rr_2=\rr(C_2,\F^\infty)\subset A_2$, $\RR_1=germ(\rr_1)$ and $\RR_2=germ(\rr_2)$. By (MA3) there exists an apartment $A_3$ containing $\RR_1$ and $\RR_2$.

 Let us prove that $A_3$ contains $x_1$ and $x_2$. One identifies $A_1$ and $\A$. Let $F^v=0+\mathfrak{F}^\infty$ and $F'^v=0+\mathfrak{F}'^\infty$. As $A_3\supset \RR_1$, there exists $f'\in F'^v$ such that $A_3\supset x_1+ f'+F'^v$. Moreover $A_3\supset \mathfrak{F}^\infty$ and thus it contains $x_1+f'+\mathfrak{F}^\infty$. By Proposition 4.7.1 of \cite{rousseau2011masures} $x_1+f'+\mathfrak{F}^\infty=x_1+f'+F^v$ and thus $A_3\ni x_1$. As $A_3\supset \RR_2$, there exists $f\in F^v$ such that $A_3\supset x_2+f$. As $A_3\supset \mathfrak{F}'^\infty$, \[A_3\supset x_2+f+\mathfrak{F}'=x_2+f+F'^v\] by Proposition 4.7.1 of \cite{rousseau2011masures}. Thus $A_3\ni x_2$.
 
If $i\in \{1,2\}$, each element of $\RR_i$ has nonempty interior in $A_i$ and thus $A_i\cap A_3$ has nonempty interior.  By hypothesis, $A_1\cap A_3$ and $A_2\cap A_3$ are convex. By Proposition~\ref{propIsomorphisms fixant une partie convexe}, there exist $\phi:A_1\overset{A_1\cap A_3}{\rightarrow} A_3$ and $\psi:A_2\overset{A_2\cap A_3}{\rightarrow} A_3$. Therefore $[x_1,x_2]_{A_1}=[x_1,x_2]_{A_3                                         }=[x_1,x_2]_{A_2}$ and thus $A_1\cap A_2$ is convex. 

The existence of an isomorphism $A_1\overset{A_1\cap A_2}{\rightarrow }A_2$ is a consequence of Proposition~\ref{propIsomorphisms fixant une partie convexe} because  the direction $X$ of $A_1\cap A_2$ meets  $\mathring{\T}$ and thus $\vec{X}\cap \T$ spans $\T$.

 The fact that $A_1\cap A_2$ is enclosed is a consequence of  Proposition~\ref{propIntersection enclose quand convexe}.
\end{proof}

\subsection{Definition of the frontier maps}\label{subsubDefinition des applications frontieres}
The aim of~\ref{subsubDefinition des applications frontieres} to~\ref{subsubConvexite sous FrC} is to prove that if $A$ and $B$ are two apartments containing a generic ray and such that $A\cap B$ has nonempty interior, then $A\cap B$ is convex.  There is no loss of generality in assuming that $B=\A$ and that the direction $\R_+\nu$ of $\delta$ is contained in $\pm \overline{C^v_f}$. As the roles of $C^v_f$ and $-C^v_f$ are similar, one supposes that $\R_+\nu \subset \overline{C^v_f}$ and that $A\neq \A$. These hypothesis run until the end of~\ref{subsubConvexite sous FrC}.

In this subsection, we parametrize $\Fr(A\cap \A)$ by a map and describe $A\cap \A$ using the values of this map.

\begin{lem}\label{lemÉlements grands dans la demi-droite generique}
Let $V$ be a bounded subset of $\A$. Then there exists $a\in \R$ such that for all $u\in [a,+\infty[$ and $v\in V$, $v\leq u\nu$. 
\end{lem}

\begin{proof}
Let $a\in \R^*_+$ and $v\in V$, then $a\nu-v=a(\nu-\frac{1}{a}v)$. As $\nu\in \mathring{\T}$ and $V$ is bounded, there exists $b>0$ such that for all $a>b$, $\nu-\frac{1}{a}v \in \mathring{\T}$, which proves the lemma because $\mathring{\T}$ is a cone.
\end{proof}

\begin{lem}\label{lemIntersection contenant des demi-droites}
Let $y\in A\cap \A$. Then $A\cap \A$ contains $y+\R_+\nu$. 
\end{lem}

\begin{proof}
Let $x\in \A$ such that $A\cap \A \supset x+\R_+\nu$. The ray $x+\R_+\nu$ is generic and by (MA4), if $y\in \A$, $A\cap \A$ contains the convex hull of $y$ and $x+[a,+\infty[\nu$, for some $a\gg 0$. In particular it contains $y+\R_+\nu$.
\end{proof}

Let $U=\{y\in \A|y+\R\nu \cap A\neq \emptyset\}=(A\cap \A)+\R\nu$.   

\begin{lem}\label{lemConvexite de U}
The set $U$ is convex.
\end{lem}

\begin{proof}
Let $u,v\in U$. Let $u'\in u+\R_+\nu\cap A$. By Lemma~\ref{lemÉlements grands dans la demi-droite generique} and Lemma~\ref{lemIntersection contenant des demi-droites}, there exists $v'\in v+\R_+\nu$ such that $u'\leq v'$. By consequence 2) of Proposition 5.4 of \cite{rousseau2011masures}, $[u',v']\subset A\cap \A$. By definition of $U$, $[u',v']+\R\nu\subset U$ and in particular $[u,v]\subset U$, which is the desired conclusion.
\end{proof}

There are two possibilities: either   there exists $y\in \A$ such that $y+\R\nu\subset A$ or  for all $y\in \A$, $y+\R\nu\nsubseteq A$. The first case is the easiest and we treat it in the next lemma. 

\begin{lem}\label{lemCas où l'intersection contient une droite generique}
Suppose that for some $y\in \A$, $y-\R_+\nu \subset A\cap \A$. Then $A\cap \A=U$. In particular, $A\cap \A$ is convex.
\end{lem} 

\begin{proof}
By Lemma~\ref{lemIntersection contenant des demi-droites}, $A\cap \A=(A\cap \A)+\R_+\nu$. By symmetry and by hypothesis on $A\cap \A$, one has $(A\cap \A)+\R_-\nu =A\cap \A$. Therefore $A\cap \A= (A\cap \A)+\R\nu= U$.
\end{proof}

\subsubsection*{Definition of the frontier} Let $u\in U$. Then by Lemma~\ref{lemIntersection contenant des demi-droites}, $u+\R\nu \cap A$ is of the form $a+\R^*_+\nu$ or $a+\R_+\nu$ for some $a\in \A$. As $A\cap \A$ is closed (by Proposition~\ref{propIntersection fermee}), the first case cannot occur. One sets $\Fr_\nu(u)=a\in \A\cap A$. One fixes $\nu$ until the end of~\ref{subsubConvexite sous FrC} and one writes $\Fr$ instead of $\Fr_\nu$.

\begin{lem}\label{lemDescription explicite de l'intersection de deux apparts}
The map $\Fr$ takes it values in  $\Fr(\A\cap A)$ and $\A\cap A=\bigcup_{x\in U}\mathrm{Fr}(x)+\R_+\nu$.
\end{lem}

\begin{proof}
Let $u\in U$. Then $\Fr(u)+\R_+\nu=(u+\R\nu)\cap A$. Thus $\Fr(u)\notin \In(A\cap \A)$. By Proposition~\ref{propIntersection fermee}, $\Fr(u)\in \Fr(A\cap \A)$ and hence $\Fr(U)\subset \Fr(A\cap \A)$.

Let $u\in A\cap \A$. One has $u\in A\cap (u+\R\nu)=\Fr(u)+\R_+\nu$ and we deduce that $\A\cap A\subset \bigcup_{x\in U}\mathrm{Fr}(x)+\R_+\nu$. The reverse inclusion is a consequence of Lemma~\ref{lemIntersection contenant des demi-droites}, which finishes the proof.
\end{proof}

Let us sketch the proof of the convexity of $A\cap \A$ (which is Lemma~\ref{lemConvexite de l'intersection cas nonvide}).  If $x,y\in \mathring{U}$, one defines $\Fr_{x,y}:[0,1]\rightarrow \Fr(A\cap \A)$ by $\Fr_{x,y}(t)=\Fr((1-t)x+ty)$ for all $t\in [0,1]$. For all $t\in [0,1]$,  there exists a unique $f_{x,y}(t)\in\R$ such that $\Fr_{x,y}(t)=(1-t)x+ty+f_{x,y}(t)\nu$.  We prove that for ``almost'' all $x,y\in \mathring{U}$, $f_{x,y}$ is convex. Let $x,y\in \mathring{U}$. We first prove that $f_{x,y}$ is continuous and piecewise affine. This enables to reduce the study of the convexity of $f_{x,y}$ to the study of $f_{x,y}$ at the points where the slope changes. Let $\mathcal{M}$ be a finite set of walls such that $\Fr(\mathring{U})\subset \bigcup_{M\in \mathcal{M}}M$, which exists by Proposition~\ref{propÉcriture de l'intersection comme union des Pi, cas general}. Using order-convexity, we prove that if $\{x,y\}$ is such that for each point $u\in ]0,1[$ at which the slope changes, $\Fr_{x,y}(u)$ is contained in exactly two walls of $\mathcal{M}$, then $f_{x,y}$ is convex. We then prove that there are ``enough'' such pairs and conclude by an argument of density.

 \subsection{Continuity of the frontier}
In this subsection, we prove that $\Fr$ is continuous on $\mathring{U}$, using order-convexity. 

Let $\lambda:U\rightarrow \R$ such that for all $x\in U$, $\Fr(x)=x+\lambda(x)\nu$. We prove the continuity of $\Fr_{|\mathring{U}}$ by proving the continuity of $\lambda_{|\mathring{U}}$. For this, we begin by dominating $\lambda([x,y])$ if $x,y\in \mathring{U}$ (see Lemma~\ref{lemBorne de lambda sur un segment}) by a number depending on $x$ and $y$. We use it to prove that if $n\in \N$ and  $a_1,\ldots,a_n\in \mathring{U}$, then $\lambda\big(\conv(\{a_1,\ldots,a_n\})\big)$ is dominated and then deduce that $\Fr_{|\mathring{U}}$ is continuous (which is Lemma~\ref{lemContinuite de la frontiere}).

\begin{lem}\label{lemBorne de lambda sur un segment}
Let $x,y\in U$, $M=\max\{\lambda(x),\lambda(y)\}$ and $k\in \R_+$ be such that $x+k\nu \geq y$. Then for all $u\in [x,y]$, $\lambda(u)\leq k+M$.
\end{lem}

\begin{proof}
By Lemma~\ref{lemIntersection contenant des demi-droites}, $x+M\nu$ and $y+M\nu$ are in $A$. By hypothesis, $x+k\nu +M\nu\geq y+M\nu$. Let $t\in [0,1]$ and $u=tx+(1-t)y$. By order-convexity $t(x+k\nu+M\nu)+(1-t)(y+M\nu)\in A$. Therefore $\lambda(u)\leq M+tk\leq M+k$, which is our assertion.
\end{proof}

\begin{lem}\label{lemCaratheodory}
Let $d\in \N$, $X$ be a $d$ dimensional affine space and $P\subset X$. One sets $\conv_0(P)=P$ and for all $k\in \N$, \[\conv_{k+1}(P)=\{(1-t)p+tp'| t\in [0,1]\mathrm{\ and\ }(p,p')\in \conv_k(P)^2\}.\] Then $\conv_d(P)=\conv(P)$.
\end{lem}

\begin{proof}
By induction, \[\conv_k(P)=\{\sum_{i=1}^{2^k} \lambda_i p_i|(\lambda_i)\in [0,1]^{2^k},\ \sum_{i=1}^{2^k}\lambda_i=1\mathrm{\ and\ }(p_i)\in P^{2^k}\}.\] This is thus a consequence of Carath\'eodory's Theorem. 
\end{proof}

\begin{lem}\label{lemBorne de lambda sur une preenveloppe convexe}
Let $P$ be a bounded subset of $\mathring{U}$ such that $\sup\big(\lambda(P)\big)<+\infty$.  Then $\sup\big(\lambda(\conv_1(P))\big)<+\infty$.
\end{lem}

\begin{proof}
Let $M=\sup_{x\in P}\lambda(x)$ and $k\in \R_+$ such that for all $x,x'\in P$,  $x'+k\nu \geq x$, which exists by Lemma~\ref{lemÉlements grands dans la demi-droite generique}. Let $u\in \conv_1(P)$ and $x,x'\in P$ such that $u\in [x,x']$. By Lemma~\ref{lemBorne de lambda sur un segment}, $\lambda(u)\leq k+M$ and the lemma follows.

\end{proof}

\begin{lem}\label{lemBorne de lambda sur un voisinage d'un point}
Let $x\in \mathring{U}$. Then there exists $V\in \V_{\mathring{U}}(x)$ such that $V$ is convex and $\sup\big(\lambda(V)\big)<+\infty$.
\end{lem}

\begin{proof}
Let $n\in \N$ and $a_1,\ldots,a_n\in \mathring{U}$ such that $V=\conv(a_1,\ldots,a_n)$ contains $x$ in its interior. Let $M\in \R_+$ such that for all $y,y'\in V$, one has $y+M\nu\geq y'$, which is possible by Lemma~\ref{lemÉlements grands dans la demi-droite generique}. One sets $P=\{a_1,\ldots,a_n\}$ and for all $k\in \N$, $P_k=\conv_k(P)$. By induction using Lemma~\ref{lemBorne de lambda sur une preenveloppe convexe}, $\sup\big(\lambda(P_k)\big)<+\infty$ for all $k\in \N$ and we conclude with Lemma~\ref{lemCaratheodory}.
\end{proof}

\begin{lem}\label{lemInegalite de continuite sur les ouverts convexes}
Let $V\subset \mathring{U}$ be open, convex, bounded and such $\sup\big(\lambda(V)\big)\leq M$ for some $M\in \R_+$. Let $k\in \R_+$ such that for all $x,x'\in V$, $x+k\nu\geq x'$. Let $a\in V$ and $u\in \A$ such that $a+u\in V$. Then for all $t\in [0,1]$, $\lambda(a+tu)\leq (1-t)\lambda(a)+t(M+k)$. 
\end{lem}

\begin{proof}
By Lemma~\ref{lemIntersection contenant des demi-droites},  $a+u+(M+k)\nu\in A$. Moreover $a+u+(M+k)\nu\geq a+M\nu$, $a+M\nu\geq a+\lambda(a)\nu=\Fr(a)$ and thus $a+u+(M+k)\nu \geq \Fr(a)$. 

Let $t\in [0,1]$. Then by order-convexity, \[(1-t)(a+\lambda(a)\nu)+t(a+u+(M+k)\nu=a+tu+\big((1-t)\lambda(a)+t(M+k)\big)\nu \in A.\]

Therefore $\lambda(a+tu)\leq (1-t)\lambda(a)+t(M+k)$, which is our assertion.
\end{proof}

\begin{lem}\label{lemContinuite de la frontiere}
The map $\Fr$ is continuous on $\mathring{U}$.
\end{lem}

\begin{proof}
Let $x\in \mathring{U}$ and $V\in \V_{\mathring{U}}(x)$ be convex, open, bounded and such that $\sup \big(\lambda(V)\big)\leq M$  for some $M\in \R_+$, which exists by Lemma~\ref{lemBorne de lambda sur un voisinage d'un point}. Let $k\in \R_+$ such that for all $v,v'\in V$, $v+k\nu \geq v'$. Let $|\ |$ be a norm on $\A$ and $r>0$ such that $B(x,r)\subset V$, where $B(x,r)=\{u\in \A|\ |x-u|\leq r\}$. Let $S=\{u\in \A|\ |u-x|=r\}$. Let $N=M+k$.

Let $y\in S$ and $t\in [0,1]$. By applying Lemma~\ref{lemInegalite de continuite sur les ouverts convexes} with $a=x$ and $u=y-x$, we get that \[\lambda((1-t)x+ty)\leq \lambda(x)+tN.\] By aplying Lemma~\ref{lemInegalite de continuite sur les ouverts convexes} with $a=(1-t)x+ty$ and $u=x-y$, we obtain that \[\lambda(x)=\lambda\big((1-t)x+ty+t(x-y)\big)\leq \lambda\big((1-t)x+ty\big)+tN.\] Therefore for all $t\in [0,1]$ and $y\in S$, \[\lambda(x)-tN\leq \lambda\big((1-t)x+ty\big) \leq \lambda(x)+tN.\]

Let $(x_n)\in B(x,r)^\N$ such that $x_n\rightarrow x$. Let $n\in \N$. One sets $t_n=\frac{|x_n-x|}{r}$. If $t_n=0$, one chooses $y_n\in S$. It $t_n\neq 0$, one sets $y_n=x+\frac{1}{t_n}(x_n-x)\in S$. Then $x_n=t_ny_n+(1-t_n)x$ and thus $|\lambda(x_n)-\lambda(x)|\leq t_n N \rightarrow 0$. Consequently $\lambda_{|\mathring{U}}$ is continuous and we deduce that $\Fr_{|\mathring{U}}$ is continuous.
\end{proof}

\subsection{Piecewise affineness of $\Fr_{x,y}$}

We now study the map $\Fr$. We begin by proving that there exists a finite set $\HH$ of hyperplanes of $\A$ such  $\Fr$ is affine on each connected component of $\mathring{U}\backslash \bigcup_{H\in \HH}H$.

Let $\M$ be finite set of walls such that $\Fr(A\cap \A)$ is contained in $\bigcup_{M\in \mathcal{M}}M$, whose existence is provided by Proposition~\ref{propÉcriture de l'intersection comme union des Pi, cas general}. Let $r=|\mathcal{M}|$. Let $\{\beta_1,\ldots,\beta_r\}\in \Phi^r$ and $(\ell_1,\ldots,\ell_r)\in \prod_{i=1}^r \Lambda'_{\beta_i}$ be such that $\mathcal{M}=\{M_i|\ i\in \llbracket 1,r\rrbracket\}$ where $M_i=\beta_i^{-1}(\{\ell_i\})$ for all $i\in \llbracket 1,r\rrbracket$.

Let  $i,j\in \llbracket 1,r\rrbracket$ be such that $i\neq j$. If $\beta_i(\nu)\beta_j(\nu)\neq 0$ and $M_i$ and $M_j$ are not parallel, one sets $ H_{i,j}=\{x\in \A|\frac{\ell_i-\beta_i(x)}{\beta_i(\nu)}=\frac{\ell_j-\beta_j(x)}{\beta_j(\nu)}\}$ (this definition will appear naturally in the proof of the next lemma). Then $H_{i,j}$ is a hyperplane of $\A$. Indeed, otherwise $H_{i,j}=\A$. Hence  $\frac{\beta_j(x)}{\beta_j(\nu)}-\frac{\beta_i(x)}{\beta_i(\nu)}=\frac{\ell_j}{\beta_j(\nu)}-\frac{\ell_i}{\beta_i(\nu)}$, for all $x\in \A$. Therefore $\frac{\beta_j(x)}{\beta_j(\nu)}-\frac{\beta_i(x)}{\beta_i(\nu)}=0$ for all $x\in \A$ and thus $M_i$ and $M_j$ are parallel: a contradiction.  Let $\HH=\{H_{i,j}|i\neq j\mathrm,\ \beta_i(\nu)\beta_j(\nu)\neq 0\mathrm{\ and\ }M_i \nparallel M_j \}\cup \{M_i|\beta_i(\nu)=0\}$.
 
 \medskip
 
  Even if the elements of $\HH$ can be walls of $\A$, we will only consider them as hyperplanes of $\A$. To avoid confusion between elements of $\M$ and elements of $\HH$, we will try to use the letter M (resp. H) in the name of objects related to $\M$ (resp. $\HH$).

\begin{lem}\label{lemFrontiere sans pb sauf hyperplans}
Let $M_\cap=\bigcup_{M\neq M'\in \M}M\cap M'$.  Then $\Fr^{-1}(M_\cap)\subset \bigcup_{H\in \mathcal{H}}H$.
\end{lem}

\begin{proof}
Let $x\in \Fr^{-1}(M_\cap)$. One has $\Fr(x)=x+\lambda\nu$, for some $\lambda\in \R$. There exists $i,j\in \llbracket 1,r\rrbracket$ such that: \begin{itemize}

\item $i\neq j$,

\item $\beta_i(\Fr(x))=\ell_i$ and $\beta_j(\Fr(x))=\ell_j$,

\item $M_i$ and $M_j$ are not parallel.

\end{itemize}  Therefore if $\beta_i(\nu)\beta_j(\nu)\neq 0$, then $\lambda=\frac{\ell_i-\beta_i(x)}{\beta_i(\nu)}=\frac{\ell_j-\beta_j(x)}{\beta_j(\nu)}$ and thus $x\in H_{i,j}$. If $\beta_i(\nu)\beta_j(\nu)=0$, then $x\in M_i\cup M_j$, which proves the lemma.
\end{proof} 

\begin{lem}\label{lemIntersection adherence de son int}
One has $A\cap \A=\overline{\In(A\cap \A)}$.
\end{lem}

\begin{proof}
By Proposition~\ref{propIntersection fermee}, $A\cap \A$ is closed and thus $\overline{\In(A\cap \A)}\subset A\cap \A$.

Let $x\in A\cap \A$. Let $V$ be an open bounded set contained in $A\cap \A$. By Lemma~\ref{lemÉlements grands dans la demi-droite generique} applied to $x-V$, there exists $a>0$ such that for all $v\in V$, one has $v+a\nu \geq x$. One has  $V+a\nu\subset A\cap \A$ and by order convexity (Cons\'equence 2 of Proposition 5.4 in \cite{rousseau2011masures}), $\conv(V+a\nu, x)\subset A\cap \A$. As $\conv(V+a\nu,x)$ is a convex set with nonempty interior, there exists $(x_n)\in \In(\conv(V+a\nu,x))^\N$ such that $x_n\rightarrow x$, and the lemma follows.
\end{proof}

Let $f_1,\ldots,f_s$ be affine forms on $\A$ such that $\HH=\{ f_i^{-1}(\{0\})|i\in \llbracket 1,s\rrbracket \}$ for some $s\in \N$. Let $R=(R_i)\in \{\leq, \geq,<,> \}^{s}$.  
One sets \[P_R=\mathring{U}\cap\{x\in \A|\  (f_i(x)\ R_i \ 0) \  \forall i\in \llbracket 1,s\rrbracket\}.\]
 If  $R=(R_i)\in \{\leq,\geq\}^s$,
  one defines $R'=(R'_i)\in \{<,>\}^{s}$ 
 by $R'_i=`` < $'' if $R_i=``\leq$'' and $R'_i=``>$'' otherwise (one replaces large inequalities by strict inequalities). If $R\in \{\leq,\geq \}^s$, then $\In({P_R})=P_{R'}$.

Let $X=\{R\in \{\leq ,\geq\}^s| \mathring{P_R}\neq \emptyset\}$. By Lemma~\ref{lemIntersection adherence de son int} and Lemma~\ref{lemUnion des Pi dont l'interieur est non vide},  $\mathring{U}=\bigcup_{R\in X} P_R$ and for all $R\in X$, $\mathring{P_R}\subset \A\backslash \bigcup_{H\in \HH}H$.

\begin{lem}\label{lemFrontiere affine pm}
Let $R\in X$. Then there exists $M\in \mathcal{M}$ such that $\Fr(P_R)\subset M$.
\end{lem}

\begin{proof}
Let $x\in \mathring{P}_{R}$. Let $M\in \mathcal{M}$ be such that $\Fr(x)\in M$.  Let us show that $\Fr(P_R)\subset M$. By continuity of $\Fr$ (by Lemma~\ref{lemContinuite de la frontiere}), it suffices to prove that $\Fr(\mathring{P}_R)\subset M$. By connectedness of $\mathring{P}_R$ it suffices to prove that $\Fr^{-1}(M)\cap \mathring{P}_R$ is open and closed. As $\Fr$ is continuous, $\Fr^{-1}(M)\cap \mathring{P}_R$ is closed (in $\mathring{P}_R$). 

Suppose that $\Fr^{-1}(M)\cap \mathring{P}_R$ is not open. Then there exists $y\in \mathring{P}_R$ such that $\Fr(y)\in M$ and a sequence $(y_n)\in (\mathring{P}_R)^\N$ such that  $y_n\rightarrow y$ and such that $\Fr(y_n)\notin M$ for all $n\in \N$. For all $n\in \N$, $\Fr(y_n)\in \bigcup_{M'\in \M} M'$, and thus, maybe extracting a subsequence, one can suppose that for some $M'\in \M$, $y_n\in M'$ for all $n\in \N$.

As $\Fr$ is continuous (by Lemma~\ref{lemContinuite de la frontiere}), $\Fr(y)\in M'$. Thus $\Fr(y)\in M\cap M'$ and by Lemma~\ref{lemFrontiere sans pb sauf hyperplans}, $y\in \bigcup_{H\in \HH}H$, which is absurd by choice of $y$. Therefore, $\Fr^{-1}(M)\cap \mathring{P}_R$ is open, which completes the proof of the lemma.
\end{proof}

\begin{lem}\label{lemFrontiere restreinte aux Pr}
Let $R\in X$ and $M\in \M$ be such that $\Fr(P_R)\subset M$. Then $\nu\notin \vec{M}$ and there exists a (unique) affine morphism $\psi:\A\rightarrow M$ such that $\Fr_{|P_R}=\psi_{|P_R}$. Moreover $\psi$ induces an isomorphism $\overline{\psi}:\A/\R\nu \rightarrow M$.
\end{lem}

\begin{proof}
If $y\in \mathring{U}$, then $\Fr(y)=y+k(y)\nu$ for some $k(y) \in \R$. Let $\alpha\in \Phi$ be such that $M=\alpha^{-1}(\{u\})$ for some $u\in -\Lambda_\alpha'$. For all $y\in P_R$, one has $\alpha(\Fr(y))=\alpha(y)+k(y)\alpha(\nu)=u$ and $\alpha(\nu)\neq 0$ because $\alpha$ is not constant on $P_{R}$. Consequently $\nu\notin \vec{M}$ and  $\Fr(y)=y+\frac{u-\alpha(y)}{\alpha(\nu)}\nu$. One defines $\psi:\A\rightarrow M$ by $\psi(y)=y+\frac{u-\alpha(y)}{\alpha(\nu)}\nu$ for all $y\in \A$ and $\psi$ has the desired properties.
\end{proof}

\subsection{Local convexity of $\Fr_{x,y}$}

Let $M\in \mathcal{M}$ and $\vec{M}$ be its direction. Let $\T_M=\mathring{\T}\cap \vec{M}$ and $D_M$ be the half-apartment containing a shortening of $\R_+\nu$ and whose wall is $M$.

\begin{lem}\label{lemInclusion dans un demi appart sur le cone}
Let $a\in \Fr(\mathring{U})$ and suppose that there exists $\mathcal{K}\in \V_{\mathring{U}}(a)$ such that  $\Fr(\mathcal{K})\subset M$ for some $M\in \mathcal{M}$. Then $\Fr\big((a\pm \mathring{\T}_M)\cap \mathring{U}\big)\subset D_M$.
\end{lem}

\begin{proof} 
Let $u\in \mathring{U}\cap (a-\mathring{\T}_M)$, $u\neq a$. Suppose $\Fr(u)\notin D_M$. Then $\Fr(u)=u-k\nu$, with $k\geq 0$. Then $\Fr(u)\leq u\mathring{\leq} a$ (which means that $a-u\in \mathring{\T}$). Therefore for some $\mathcal{K}'\in \V_M(a)$ such that $\mathcal{K}'\subset \mathcal{K}$, one has $\Fr(u)\mathring{\leq} u'$ for all $u'\in \mathcal{K}'$. As a consequence $\A\cap A\supset \conv(\mathcal{K}',\Fr(u))$ and thus $\Fr(u')\notin M$ for all $u'\in \mathcal{K}'$. This is absurd and hence $\Fr(u)\in D_M$.

Let $v\in \mathring{U}\cap(a+\mathring{\T}_M) $, $v\neq a$ and suppose that $\Fr(v)\notin D_M$. Then for $v'\in [\Fr(v),v[$ near enough from $v$, one has $a\leq v'$. Therefore, $[a,v']\subset \A\cap A$. Thus for all $t\in ]a,v[$, $\Fr(t)\notin D_M$, a contradiction. Therefore $\Fr(v)\in D_M$ and the lemma follows.
\end{proof}

The following lemma is crucial to prove the local convexity of $\Fr_{x,y}$  for good choices of $x$ and $y$. This is mainly here that we use  that $A\cap \A$ have nonempty interior.

Let $H_\cap=\bigcup_{H\neq H'\in \HH}  H\cap H'$.

\begin{lem}\label{lemConcavite locale}
Let $x\in \mathring{U}\cap(\bigcup_{H\in \HH}H )\backslash H_\cap$ and $H\in \HH$ be such that $x\in H$. Let $C_1$ and $C_2$ be the half-spaces defined by $H$. Then there exists $V\in \V_{\mathring{U}}(x)$ satisfying the following conditions:

\begin{enumerate}
\item\label{itemDecoupage de V} For $i\in \{1,2\}$, let $V_i=V\cap \mathring{C_i}$. Then $V_i\subset \mathring{P_{R_i}}$ for some $R_i\in X$.

\item\label{itemConvexite locale}  Let $M$ be a wall containing $\Fr(P_{R_1})$. Then  $\Fr(V)\subset D_{M}$.
\end{enumerate}
\end{lem}

\begin{proof} (see Figure~\ref{figConcavite locale})
The set $\mathring{U}\backslash \bigcup_{H\in \HH \backslash\{H\}}H$ is open in $\mathring{U}$. Let $V'\in \V_{\mathring{U}}(x)$ be such that 
 $V'\cap \bigcup_{H'\in \HH\backslash\{H\}}H'=\emptyset$ and such that $V'$ is convex.
  Let $i\in\{1,2\}$ and $V_i'=V'\cap \mathring{C_i}$. Then $V_i'\subset \mathring{U}\backslash \bigcup_{H\in \mathcal{H}} H$. Moreover $V_i'$ is connected. As the connected components of $\mathring{U}\backslash \bigcup_{H\in \mathcal{H}} H$ are the $\mathring{P}_{R}$'s for $R\in X$, we deduce that $V'$ satisfies~\ref{itemDecoupage de V}.

Let $\psi:\A\rightarrow M$ be the affine morphism such that $\psi_{|P_{R_1}}=\Fr_{|P_{R_1}}$ and $\overline{\psi}:\A/\R\nu\rightarrow M$ be the induced isomorphism, which exist by Lemma~\ref{lemFrontiere restreinte aux Pr}. Let $\pi:\A\rightarrow \A/\R\nu$ be the canonical projection. As $C_1$ is invariant under translation by $\nu$ (by definition of the elements of $\HH$), the set $\psi(C_1)=\overline{\psi}(\pi(C_1))$ is a half-space $D$ of $M$. Let $V''=V'\cap C_1$. Then $\psi(V'')=\overline{\psi}(C_1)\cap \overline{\psi}(\pi(V'))\in \V_{D}(\Fr(x))$. 

Let $g:\vec{M}\rightarrow \R$ be a linear form such that $D=g^{-1}([b,+\infty[)$, for some $b\in \R$. Let $\epsilon \in \{-1,1\}$ be such that $g(u)>0$  for some $u\in \epsilon\T_{M}$. Let $\eta>0$.  Then $\Fr(x+\eta u)\in x+\eta u+\R\nu$ and thus $\Fr(x+\eta u)=\Fr(x)+\eta u+k\nu$ for some $k\in \R$. If $\eta$ is small enough that  $x+\eta u\in V''$, then  $k\nu=\Fr(x+\eta u)-(\Fr(x)+\eta u)\in \vec{M}$ and hence $k=0$ (by Lemma~\ref{lemFrontiere restreinte aux Pr}). Let $\mathcal{K}=\psi(V'')+\R\nu$ and $a=\Fr(x)+\eta u$. Then $\mathcal{K}\in \V_{\mathring{U}}(a)$ and for all $v\in \mathcal{K}$, $\Fr(v)\in M$. By Lemma~\ref{lemInclusion dans un demi appart sur le cone}, \[\Fr(\mathring{U}\cap (a-\epsilon\T_M))=\Fr(\mathring{U}\cap (a-\epsilon\T_M+\R\nu))\subset D_M.\] Moreover, $a-\epsilon\T_M+\R\nu \in \V_{\mathring{U}}(x)$ and thus if one sets $V=V'\cap (a-\epsilon\T_M+\R\nu)$,  $V$ satisfies~\ref{itemDecoupage de V} and~\ref{itemConvexite locale}.

\end{proof}

\begin{figure}
\centering
\includegraphics[height=4cm]{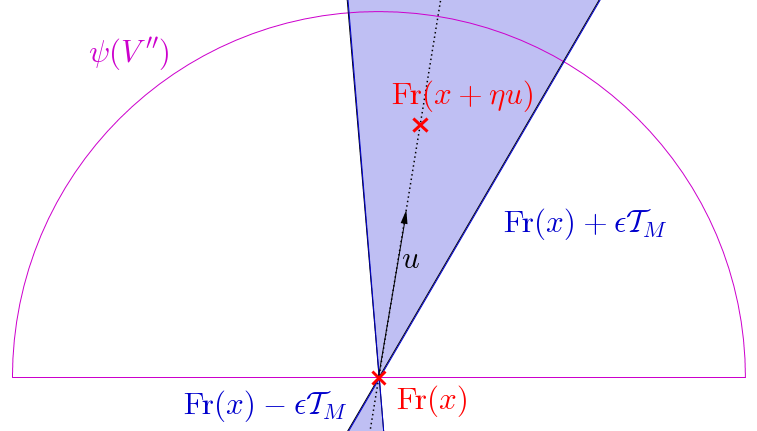}
\hspace{0.5cm}
\includegraphics[height=4cm]{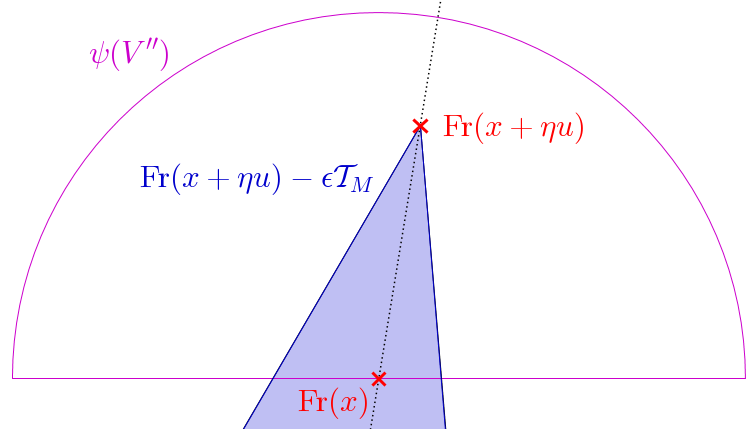}
\caption{Proof of Lemma~\ref{lemConcavite locale} when $\dim H=2$ (the illustration is made in $M$)}\label{figConcavite locale}
\end{figure}

\subsection{Convexity of $A\cap \A$}\label{subsubConvexite sous FrC}

Let $\vec{\HH}=\bigcup_{H\in {\HH}}\vec{H}$ be the set of directions of the hyperplanes of $\HH$.

\begin{lem}\label{lemConvexite pour les bonnes paires}
 Let $x,y\in \mathring{U}\cap A\cap \A$ be such that $y-x\notin \vec{\HH}$ and such that the line spanned by $[x,y]$ does not meet any point of $H_\cap$. Then $[x,y]\subset \mathring{U}\cap A\cap\A$.
\end{lem}

\begin{proof}
Let $\pi:[0,1]\rightarrow \A$ defined by $\pi(t)=tx+(1-t)y$ for all $t\in [0,1]$. Set $g=\Fr \circ \pi$. Let $f_1,\ldots,f_s$ be affine forms on $\A$ such that $\HH=\{f_i^{-1}(\{0\}|\ i\in \llbracket 1,s\rrbracket\}$. As $y-x\notin \vec{\HH}$, for all $i\in \llbracket 1,s\rrbracket$, the map $f_i\circ g$ is strictly monotonic and $\pi^{-1}(\bigcup_{H\in \HH} H)$ is finite. Therefore, there exist $k\in \N$ and open intervals $T_1\ldots,T_k$ such that: \begin{itemize}

\item $[0,1]=\bigcup_{i=1}^k \overline{T_i}$,

\item $T_1< \ldots<T_k$, 

\item for all $i\in \llbracket 1,k\rrbracket$, there exist $R_i\in X$ such that $\pi(T_i)\subset \mathring{P}_{R_i}$.

\end{itemize} For all $t\in [0,1]$, one has $g(t)=\pi(t)+f(t)\nu$ for some $f(t)\in \R$. By Lemma~\ref{lemFrontiere restreinte aux Pr} this equation uniquely determines $f(t)$ for all $t\in [0,1]$. By Lemma~\ref{lemContinuite de la frontiere}, $f$ is continuous and by Lemma~\ref{lemFrontiere restreinte aux Pr}, $f$ is affine on each $T_i$.

Let us prove that $f$ is convex. Let $i\in \llbracket 1,k-1\rrbracket$. One writes $T_i=]a,b[$. Then for $\epsilon>0$ small enough, one has $f(b+\epsilon)=f(b)+\epsilon c_+$ and $f(b-\epsilon)=f(b)-\epsilon c_-$. To prove the convexity of $f$, it suffices to prove that $c_-<c_+$. Let $M$ be a wall containing $\Fr(P_{R_i})$. As $\pi(b)\in \mathring{U}\cap \bigcup_{H\in \HH}H \backslash H_\cap $, we can apply Lemma~\ref{lemConcavite locale} and there exists $V\in \V_{[0,1]}(b)$ such that $g(V)\subset D_M$. Let $h:\A\rightarrow \R$ be a linear map such that $D_M=h^{-1}([a,+\infty[)$. For $\epsilon>0$ small enough, one has $h(g(b+\epsilon))\geq a$ and $h(g(b-\epsilon))= a$. 

For $\epsilon>0$ small enough, one has \[\begin{aligned} h(g(b+\epsilon)) &=h(\pi(b)+\epsilon(y-x)+(f(b)+\epsilon c_+)\nu)\\ &= h(g(b)+\epsilon(y-x+c_+\nu))\\ &= a+\epsilon (h(y-x)+c_+h(\nu))\geq a ,\end{aligned}\]

and similarly, $h(g(b-\epsilon))=a-\epsilon(h(y-x)+c_-h(\nu))=a$. 

Therefore $h(y-x)+c_+h(\nu) \geq 0$, $h(y-x)+c_-h(\nu)=0$ and hence $(c_+-c_-)h(\nu)\geq 0$. As $D_M$ contains a shortening of $\R_+\nu$, $h(\nu)\geq 0$ and by Lemma~\ref{lemFrontiere restreinte aux Pr}, $h(\nu)>0$. Consequently, $c_-\leq c_+$ and, as $i\in \llbracket 1,k-1\rrbracket$ was arbitrary,  $f$ is convex.

For all $t\in [0,1]$, $f(t)\leq (1-t)f(0)+tf(1)$. Therefore \[(1-t)g(0)+tg(1)=\pi(t)+((1-t)f(0)+tf(1))\nu\in \pi(t)+f(t)\nu+\R_+\nu=g(t)+\R_+\nu.\]

By definition of $\Fr$, if $t\in [0,1]$, then $(1-t)g(0)+tg(1)\in A\cap \A$. Moreover, there exist $\lambda,\mu \geq 0$ such that $x= g(0)+\lambda\nu$ and $y=g(1)+\mu\nu$. Then \[\pi(t)=(1-t)x+ty=(1-t)g(0)+tg(1)+((1-t)\lambda+t\mu)\nu\in A\cap \A\] and hence $[x,y]\subset A\cap \A$.

\end{proof}

\begin{lem}\label{lemDensite des bonnes paires de points}
Let $x,y\in \In(\A\cap A)$ and $\vec{\HH}=\bigcup_{H\in \HH}\vec{H}$. Then there exists $(x_n),(y_n)\in \In(A\cap \A)^\N$ satisfying the following conditions:

 \begin{enumerate}
\item $x_n\rightarrow x$ and $y_n\rightarrow y$

\item for all $n \in \N$, $y_n-x_n \notin \vec{\HH}$

\item the line spanned by $[x_n,y_n]$ does not meet any point of $H_\cap$.
\end{enumerate}

\end{lem}

\begin{proof}
Let $(x_n)\in (\In(A\cap\A)\backslash H_\cap)^\N$ be such that $x_n\rightarrow x$. Let $|\ |$ be a norm on $\A$. Let $n\in  \N$. Let $F$ be the set of points $z\in \A$ such that the line spanned by $[x_n,z]$ meets $H_\cap$. Then $F$ is a finite union of hyperplanes of $\A$ (because $H_\cap$ is a finite union of spaces of dimension at most $\dim \A -2$). Therefore $\A\backslash (F\cup x_n+\vec{\HH})$ is dense in $\A$ and one can choose $y_n\in \A\backslash (F\cup x_n+\vec{\HH})$ such that $|y_n-y|\leq \frac{1}{n+1}$. Then $(x_n)$ and $(y_n)$ satisfy the conditions of the lemma.
\end{proof}

\begin{lem}\label{lemConvexite de l'intersection cas nonvide}
The set $A\cap \A$ is convex.
\end{lem}

\begin{proof}
Let $x,y\in \In (A\cap \A)$. Let $(x_n),(y_n)$ be as in Lemma~\ref{lemDensite des bonnes paires de points}. Let $t\in [0,1]$. As $\In (A\cap \A)\subset \mathring{U}$,  one has  $tx_n+(1-t)y_n \in A\cap \A$ for all $n\in \N$, by Lemma~\ref{lemConvexite pour les bonnes paires}. As $A\cap \A$ is closed (by Proposition~\ref{propIntersection fermee}), $tx+(1-t)y\in A\cap\A$. Therefore $\In (A\cap \A)$ is convex. Consequently $A\cap \A=\overline{\In(A\cap \A)}$ (by Lemma~\ref{lemIntersection adherence de son int}) is convex.
\end{proof}

We thus have proved the following theorem:

\begin{thrm}\label{thmIntersection enclose fort}
Let $A$ and $B$ be two apartments sharing a generic ray. Then $A\cap B$ is enclosed and there exists an isomorphism $\phi:A\overset{A\cap B}{\rightarrow} B$.
\end{thrm}

\begin{proof}
By Lemma~\ref{lemConvexite de l'intersection cas nonvide} and Lemma~\ref{lemReduction au cas d'interieur non vide},  $A\cap B$ is convex. By Proposition~\ref{propIntersection enclose quand convexe}, $A\cap B$ is enclosed and by Proposition~\ref{propIsomorphisms fixant une partie convexe}, there exists an isomorphism $\phi:A\overset{A\cap B}{\rightarrow}B$.
\end{proof}

\subsection{A partial reciprocal}

One says that a group $G$ of automorphisms of $\I$ acts strongly transitively on $\I$ if the isomorphisms involved in (MA2) and (MA4) are induced by elements of $G$. For example if $G$ is a quasi-split Kac-Moody group over an ultrametric field $\mathcal{K}$, it acts strongly transitively on the associated masure  $\I(G,\mathcal{K})$. 

We now prove a kind of weak reciprocal of Theorem~\ref{thmIntersection enclose fort} when some group $G$ acts strongly transitively on $\I$ and when $\I$ is thick, which means that each panel is contained in at least three chambers. This implies some restrictions on $\Lambda'$ by Lemma~\ref{lemDemi-apparts comme intersections d'apparts} below and Remark~\ref{rqueSur les vrais murs}.

\begin{lem}\label{lemFrontiere et ensemble minimal definissant une partie enclose}
Let $P$ be an enclosed subset of $\A$ and suppose that $\mathring{P}\neq \emptyset$. One fixes the origin of $\A$ in some point of $\mathring{P}$. Let $j_P$ be the gauge of $P$ defined in Section~\ref{subEnclosite d'une intersection convexe}. Let $U=\{x\in \A| j_P(x)\neq 0\}$. One defines $\Fr:U\rightarrow P$ as in Lemma~\ref{lemDescription de la frontiere par jauge}. One writes $P=\bigcap_{i=1}^k D_i$, where the $D_i$'s are half-apartments of $\A$. Let $j\in \llbracket 1,k\rrbracket$, $M_j$ be the wall of $D_j$ and suppose that for all open subset $V$ of $U$, $\Fr(V)\nsubseteq M_j$. Then $P=\bigcap_{i\in \llbracket 1,k\rrbracket \backslash \{j\}} D_i$.
\end{lem}

\begin{proof}
Suppose that $ P\nsubseteq \bigcap_{i\in \llbracket 1,k\rrbracket \backslash \{j\}} D_i $. Let $V$ be a nonempty open and bounded subset contained in $ \bigcap_{i\in \llbracket 1,k\rrbracket \backslash \{j\}} D_i\backslash P$. Let $n\in \Ne$ be such that $\frac{1}{n}V\subset P$. Let $v\in V$. Then $[\frac{1}{n}v,v]\cap \Fr(P)=\{\Fr(v)\}$. Moreover for all $i\in \llbracket 1,k\rrbracket \backslash \{j\}$, $[\frac{1}{n}v,v]\subset \mathring{D}_i$. As $\Fr(P)\subset \bigcup_{i\in \llbracket 1,k\rrbracket} M_i$, we deduce that $\Fr(v)\in M_j$: this is absurd and thus $P=\bigcap_{i\in \llbracket 1,k\rrbracket \backslash \{j\}} D_i$.

\end{proof}

\begin{lem}\label{lemDemi-apparts comme intersections d'apparts}
Suppose that $\I$ is thick. Let $D$ be a half-apartment of $\A$. Then there exists an apartment $A$ of $\A$ such that $D=A\cap \A$.
\end{lem}

\begin{proof}
Let $F$ be a panel of the wall of $D$. As $\I$ is thick, there exists a chamber $C$ dominating $F$ and such that $C\nsubseteq \A$. By Proposition 2.9 1) of \cite{rousseau2011masures}, there exists an apartment $A$ containing $D$ and $C$. The set $\A\cap A$ is a half-apartment by Lemma~\ref{lemIntersection de deux apparts contenant un demi-appart} and thus $\A\cap A=D$, which proves the lemma.
\end{proof}

\begin{prop}\label{propIntersection de deux apparts}
Suppose that $\I$ is thick and that some group $G$ acts strongly transitively on $\I$. Let $P$ be an enclosed subset of $\A$ containing a generic ray and having nonempty interior. Then there exists an apartment $A$ such that $A\cap \A=P$.
\end{prop}

\begin{proof}
One writes $P=D_1\cap \ldots \cap D_k$, where the $D_i$'s are half-apartments of $\A$. One supposes that $k$ is minimal for this writing, which means that for all $i\in \llbracket 1, n\rrbracket$, $P\neq \bigcap_{j\in \llbracket 1,k\rrbracket\backslash  \{i\}} D_j$. For all $i\in \llbracket 1,n\rrbracket$, one chooses an apartment $A_i$ such that $\A\cap A_i =D_i$. Let $\phi_i:\A\overset{D_i}{\rightarrow} A_i$ and $g_i\in G$ inducing $\phi_i$.

Let $g=g_1 \ldots g_k$ and $A=g.\A$.  Then $A\cap \A\supset D_1\cap \ldots \cap D_k$ and $g$ fixes $D_1\cap \ldots \cap D_k$. Let us show that $A\cap \A=\{x\in \A| g.x=x\}$. By Theorem~\ref{thmIntersection enclose fort}, there exists $\phi:\A\overset{A\cap \A}{\rightarrow} A$. Moreover $g_{|\A}^{-1}\circ \phi:\A \rightarrow \A$ fixes $D_1\cap \ldots \cap D_k$, which has nonempty interior and thus $g_{|\A}^{-1}\circ \phi=\Id_{\A}$, which proves that $A\cap \A=\{x\in \A| g.x=x\}$.

 Suppose that $A\cap \A \supsetneq D_1\cap \ldots \cap D_k$. Let $i\in \llbracket 1,k\rrbracket$ be such that there exists $a\in A\cap \A\backslash D_i$.

One fixes the origin of $\A$ in some point of $\mathring{P}$, one sets $U=\{x\in \A|\ j_P(x)\neq 0\}$ and one defines $\Fr:U\rightarrow \Fr(P)$ as in Lemma~\ref{lemDescription de la frontiere par jauge}. By minimality of $k$ and Lemma~\ref{lemFrontiere et ensemble minimal definissant une partie enclose}, there exists a nonempty open set $V$ of $U$ such that $\Fr(V)\subset M_i$.

By the same reasoning as in the proof of Lemma~\ref{lemInclusion dans un demi-apart}, $\Fr(V)\cap M_i$ is open in $M_i$.  Consequently, there exists $v\in \Fr(V)$ such that $v\notin \bigcup_{j\in \llbracket 1,k\rrbracket \backslash \{i\}} M_j$. Let $V'\in \V_U(v)$ be such that $V'\cap \bigcup_{j\in \llbracket 1,k\rrbracket \backslash \{i\}} M_j=\emptyset$ and such that $V'$ is convex. Then $V'\subset \bigcap_{j\in \llbracket 1,k\rrbracket\backslash \{i\}} \mathring{D_j}$. Let  $V''=\Fr(V)\cap V'$. By Theorem~\ref{thmIntersection enclose fort}, $[a,v]\subset A\cap \A$ and hence $g$ fixes $[a,v]$. Moreover for $u\in [a,v]$ near $v$, one has $u\in \bigcap_{j\in \llbracket 1,k\rrbracket\backslash \{i\}} D_j$. Then $g.u=g_1\ldots g_i. (g_{i+1}\ldots g_k.u)=g_1\ldots g_i.u$. Moreover, $g_i.u=g_{i-1}^{-1}.\ldots.g_1^{-1}.u=u$. Therefore $u\in D_i$, which is absurd by choice of $u$.
\end{proof}

\begin{remark}\label{rqueSur les intersections d'apparts}
\begin{enumerate}

\item In the proof above, the fact that $P$ contains a generic ray is only used to prove that $A\cap \A$ is convex and that there exists an isomorphism $\phi:A\overset{A\cap \A}{\rightarrow} \A$. When $G$ is an affine Kac-Moody group and $\I$ is its masure, we will see that these properties are true without assuming that $A\cap \A$ contains a generic ray. Therefore, for any enclosed subset $P$ of $\A$ having nonempty interior, there exists an apartment $A$ such that $A\cap \A=P$

\item Let $\mathbb{T}$ be a discrete homogeneous tree with valence $3$ and $x$ be a vertex of $\mathbb{T}$. Then there exists no pair $(A,A')$ of apartments such that $A\cap A'=\{x\}$. Indeed, let $A$ be an apartment containing $x$ and $C_1,C_2$ be the alcoves of $A$ dominating $x$. Let $A'$ be an apartment containing $x$. If $A'$ does not contain $C_1$, it contains $C_2$ and thus $A\cap A'\neq \{x\}$. Therefore the hypothesis ``$P$ has nonempty interior'' is necessary in Proposition~\ref{propIntersection de deux apparts}.

\end{enumerate}

\end{remark}

\section{Axiomatic of masures}\label{secAxiomatique des masures}

\subsection{Axiomatic of masures in the general case}
The aim of this section is to give an other axiomatic of masure than the one of \cite{rousseau2011masures} and \cite{rousseau2017almost}. For this, we mainly use Theorem~\ref{thmIntersection enclose fort}.

We fix an apartment $\A=(\mathcal{S},W,\Lambda')$.  A \textbf{construction} of type $\A$ is a set endowed with a covering of subsets called apartments and satisfying (MA1).

\medskip

Let $\cl\in \mathcal{CL}_{\Lambda'}$. Let (MA i)=(MA1). 

Let (MA ii) : if two apartments $A,A'$ contain a generic ray, then $A\cap A'$ is enclosed and there exists an isomorphism $\phi:A\overset{A\cap A'}{\rightarrow} A'$.

Let (MA iii, $\cl$): if $\RR$ is the germ of a splayed chimney and if $F$ is a face or a germ of a chimney, then there exists an apartment containing $\RR$ and $F$.

\medskip

It is easy to see that the axiom (MA ii) implies (MA4, cl) for all $\cl\in \mathcal{CL}_{\Lambda'}$. If $\cl \in \mathcal{CL}_{\Lambda'}$, then  (MA iii, $\cl$) is equivalent to (MA3, $\cl$) because each chimney is contained in a solid chimney.

\medskip

Let $\I$ be a construction of type $\A$ and $\cl\in \mathcal{CL}_{\Lambda'}$. One says that $\I$ is a \textbf{masure of type }$(1,\cl)$ if it satisfies the axioms of \cite{rousseau2011masures}:  (MA2, $\cl$), (MA3, $\cl$), (MA4, $\cl$) and (MAO). One says that $\I$ is a \textbf{masure of type }$(2,\cl)$ if it satisfies  (MA ii) and (MA iii, $\cl$).

The aim of the next two subsections is to prove the following theorem: 

\begin{thrm}\label{thmAxiomatique des masures}
Let $\I$ be a construction of type $\A$ and $\cl\in \mathcal{CL}_{\Lambda'}$. Then $\I$ is a masure of type $(1,\cl)$ if and only if $\I$ is a masure of type $(1,\cl^\#)$ if and only if $\I$ is a masure of type $(2,\cl)$ if and only if $\I$ is a masure of type $(2,\cl^\#)$.
\end{thrm}

Let us introduce some other axioms and definitions. An \textbf{extended chimney} of $\A$ is associated to a local face $F^l=F^\ell(x,F_0^v)$ (its \textbf{basis}) and a vectorial face (its \textbf{direction}) $F^v$, this is the filter $\mathfrak{r}_{e}(F^\ell,F^v)=F^\ell+F^v$. Similarly to classical chimneys, we define shortenings and germs of extended chimney. We use the same vocabulary for extended chimneys as for classical: splayed, solid, full, ... We use the isomorphisms of apartments to extend these notions in constructions. Actually each classical chimney is of the form $\cl(\mathfrak{r}_e)$ for some extended chimney $\mathfrak{r}_e$.

\medskip

Let $\cl\in \mathcal{CL}_{\Lambda'}$. Let (MA2', $\cl$): if $F$ is a point, a germ of a preordered interval or a splayed chimney in an apartment $A$ and if $A'$ is another apartment containing $F$ then $A\cap A'$ contains the enclosure $\cl_A(F)$ of $F$ and there exists an isomorphism from $A$ onto $A'$ fixing $\cl_A(F)$.

Let (MA2'', $\cl$): if $F$ is a solid chimney in an apartment $A$ and if $A'$ is an other apartment containing $F$ then $A\cap A'$ contains the enclosure $\cl_A(F)$ of $F$ and there exists an isomorphism from $A$ onto $A'$ fixing $\cl_A(F)$.

The axiom (MA2, $\cl$) is a consequence of (MA2', $\cl$), (MA2'', $\cl$) and (MA ii).

\medskip

Let (MA iii'): if $\RR$ is the germ of a splayed extended chimney and if $F$ is a local face or a germ of an extended chimney, then there exists an apartment containing $\RR$ and $F$.
\medskip 

Let  $\I$ be a construction. Then $\I$ is said to be a  \textbf{masure of type $3$} if it satisfies (MA ii) and (MA iii').

In order to prove Theorem~\ref{thmAxiomatique des masures}, we will in fact prove the following stronger theorem:

\begin{thrm}\label{thmAxiomatique des masures forte}
Let $\cl\in \mathcal{CL}_{\Lambda'}$ and $\I$ be a construction of type $\A$. Then $\I$ is a masure of type $(1,\cl)$ if and only $\I$ is a masure of type $(2,\cl)$ if and only if $\I$ is a masure of type $3$.
\end{thrm}

The proof of this theorem will be divided in two steps. In the first step, we prove that (MAO) is a consequence of variants of (MA1), (MA2), (MA3) and (MA4) (see Proposition~\ref{propSuppression de MAO} for a precise statement). This uses paths but not Theorem~\ref{thmIntersection enclose fort}. In the second step, we prove the equivalence of the three definitions. One implication relies on Theorem~\ref{thmIntersection enclose fort}.

\subsubsection{Dependency of (MAO)}
The aim of this subsection is to prove the following proposition:

\begin{prop}\label{propSuppression de MAO}
Let $\I$ be a construction of type $\A$ satisfying (MA2'), (MA iii') and (MA4). Then $\I$ satisfies (MAO).
\end{prop}

We now fix a construction $\I$ of type $\A$ satisfying (MA2'), (MA iii') and (MA4). To prove proposition above, the key step  is to prove that if $B$ is an apartment and if $x,y\in \A\cap B$ are such that $x\leq_\A y$, then the image by $\rho_{-\infty}$ of the segment of $B$ joining $x$ to $y$ is a $(y-x)^{++}$-path, where if $u\in \T$, $u^{++}$ is the unique element of $W^v.u\cap \overline{C^v_f}$.

Let $a,b\in \A$. An $(a,b)$\textbf{-path of }$\A$ is a continuous piecewise linear map $[0,1]\rightarrow \A$ such that for all $t\in [0,1[$, $\pi'(t)^+\in W^v.(b-a)$. When $a\leq b$, the $(a,b)$-paths are the $(b-a)^{++}$-paths defined in~\ref{subsubCaracterisation des points de A}.

Let $A$ be an apartment an $\pi:[0,1]\rightarrow A$ be a map. Let $a,b\in A$. One says that $\pi$ is an $(a,b)$\textbf{-path} of $A$  if there exists $\Upsilon:A\rightarrow \A$ such that $\Upsilon\circ \pi$ is a $(\Upsilon(a),\Upsilon(b))$-path of $\A$.

\begin{lem}\label{lemImage d'un lambda seg par un iso}
Let $A$ be an apartment and  $a,b\in A$. Let $\pi:[0,1]\rightarrow A$ be an $(a,b)$-path in $A$ and $f:A\rightarrow B$ be an isomorphism of apartments. Then $f\circ \pi$ is a $(f(a),f(b))$-path.
\end{lem}

\begin{proof}
Let $\Upsilon:A\rightarrow \A$ be an isomorphism such that $\Upsilon\circ \pi$ is a $(\Upsilon(a),\Upsilon(b))$-path in $\A$. Then $\Upsilon'=\Upsilon\circ f^{-1}:B\rightarrow \A$ is an isomorphism, $\Upsilon'\circ f\circ \pi$ is a $(\Upsilon'(f(a)),\Upsilon'(f(b)))$-path in $\A$ and we get the lemma.
\end{proof}

The following lemma slightly improves Proposition 2.7 1) of \cite{rousseau2011masures}. We recall that if $A$ is an affine space and $x,y\in A$, $[x,y)$ means the germ $germ_x([x,y])$, $(x,y]$ means $germ_y([x,y])$, ..., see~\ref{subNotation}.
 
\begin{lem}\label{lemProposition 2.7 de rou11}
Let $\RR$ be the germ of a splayed extended chimney, $A$ be an apartment of $\I$ and $x^-,x^+\in A$ be such that $x^-\leq _A x^+$. Then there exists a subdivision $z_1=x^-,\ldots, z_n=x^+$ of $[x^-,x^+]_A$ such that for all $i\in \llbracket 1,n-1\rrbracket$ there exists an apartment $A_i$ containing $[z_i,z_{i+1}]_A\cup \RR$ such that there exists an isomorphism $\phi_i:A\overset{[z_i,z_{i+1}]_{A_i}}{\rightarrow} A_i$. 
\end{lem}

\begin{proof}
Let $u\in [x^-,x^+]$. By (MA iii'), applied to $(x^-,u]$ and $[u,x+)$ there exist  apartments $A_u^-$ and $A_u^+$ containing $\RR\cup (x^-,u]$ and $\RR\cup [u,x^+)$ and by (MA2'), there exist isomorphisms $\phi_u^+:A\overset{(x^-,u]}{\rightarrow }A_u^-$ and $\phi_u^-:A\overset{[u,x^+)}{\rightarrow} A_u^+$. For all $u\in [x^-,x^+]$ and $\epsilon\in \{-,+\}$, one chooses a convex set 
$V_u^\epsilon\in [u,x^\epsilon)$  such that $V_u^\epsilon\subset A\cap A_u^\epsilon$   and $V_u^\epsilon$ is fixed by $\phi_u^\epsilon$. If $u\in [x^-,x^+]$, one sets $V_u=\In_{[x^-,x^+]_A} (V_u^+\cup V_u^-)$. By compactness of $[x^-,x^+]$, there exists a finite set $K$ and a map $\epsilon: K\rightarrow \{-,+\}$ such that $[x^-,x^+]=\bigcup_{k\in K} V_k^{\epsilon(k)}$ and the lemma follows.
\end{proof}

Let $\q$ be a sector-germ. Then $\q$ is an extended chimney. Let $A$ be an apartment containing $\q$. The axioms (MA2'), (MA iii') and (MA4) enable one to define a retraction $\rho:\I\overset{\q}{\rightarrow}\A$ as in 2.6 of \cite{rousseau2011masures}. 

\begin{lem}\label{lemImage d'un segment par une retraction 2}
Let $A$ and $B$ be two apartments, $\q$ be a sector-germ of $B$ and $\rho:\I\overset{\q}{\rightarrow} B$. Let $x,y\in A$ be such that $x\leq_A y$. Let $\tau:[0,1]\rightarrow A$ mapping each $t\in [0,1]$ on $(1-t)x+_Aty$ and $f:A\rightarrow B$ be an isomorphism. Then $\rho\circ \tau$ is a $(f(x),f(y))$-path of $B$.

\begin{proof}
By Lemma~\ref{lemProposition 2.7 de rou11}, there exist $k\in \N$ and $t_1=0<\ldots <t_k=1$ such that for all $i\in \llbracket 1,k-1\rrbracket$, there exists an apartment $A_i$ containing $\tau([t_i,t_{i+1})])\cup \q$ such that there exists an isomorphism $\phi_i:A\overset{\tau([t_i,t_{i+1}])}{\rightarrow} A_i$.

If $i\in \llbracket 1,k-1\rrbracket$, one denotes by $\psi_i$ the isomorphism $A_i\overset{\q}{\rightarrow} B$. Then for $t\in [t_i,t_{i+1}]$, one has $\rho(\tau(t))=\psi_i\circ \phi_i(\tau(t))$. Let $\Upsilon: B \rightarrow \A$ be an isomorphism. By (MA1), for all $i\in \llbracket 1,k\rrbracket$, there exists $w_i\in W$ such that $\Upsilon\circ \psi_i\circ \phi_i=w_i\circ \Upsilon\circ f$.

Let $i\in \llbracket 1,k-1\rrbracket $ and $t\in [t_i,t_{i+1}]$. Then \[ \Upsilon\circ \rho\circ \tau(t)=\Upsilon\circ \psi_i\circ \phi_i\circ \tau(t)=(1-t) w_i\circ \Upsilon\circ f(x)+t w_i \circ \Upsilon \circ f(y) .\]

Therefore $\rho\circ \tau$ is a $(f(x),f(y))$-path in $B$.
\end{proof}
\end{lem}

\begin{lem}\label{lemArrivee d'un Lambda chemin}
Let $\lambda\in \overline{C^v_f}$ and $\pi:[0,1]\rightarrow \A$ be a $\lambda$-path. Then $\pi(1)-\pi(0)\leq_{Q^\vee} \lambda$.
\end{lem}

\begin{proof}
By definition, there exists $k\in \N$, $(t_i)\in [0,1]^k$ and $(w_i)\in (W^v)^k$ such that $\sum_{i=1}^k t_i =1$ and $\pi(1)-\pi(0)=\sum_{i=1}^k t_i.w_i.\lambda$. Therefore $\pi(1)-\pi(0)-\lambda=\sum_{i=1}^k t_i(w_i.\lambda-\lambda)$ and thus $\pi(1)-\pi(0)-\lambda\leq _{Q^\vee} 0$ by Lemma~\ref{lem2.4 a de GR14}.
\end{proof}

\begin{lem}\label{lemImage d'un segment ordonne par une retraction}
Let $x,y\in \A$  be such that $x\leq _{\A}y$ and $B$ be an apartment containing $x,y$. Let $\tau_B:[0,1]\rightarrow B$ defined by $\tau_B(t)=(1-t)x+_B ty$. Let $\mathfrak{s}$ be a sector-germ of $\A$  and $\rho_{\mathfrak{s}}:\I\overset {\mathfrak{s}}{\rightarrow}\A$. Then $x\leq_B y$ and  $\pi_\A:= \rho_{\mathfrak{s}}\circ \tau_B$ is an $(x,y)$-path of $\A$.
\end{lem}

\begin{proof}
Maybe changing the choice of $\overline{C^v_f}$, one can suppose that $y-x\in \overline{C^v_f}$. Let $\q$ be a sector-germ of $B$, $\rho_B:\I\overset{\q}{\rightarrow} B$ and $\tau_\A:[0,1]\rightarrow \A$ defined by $\tau_\A(t)=(1-t)x+ty$. Let $\phi:\A\rightarrow B$. By Lemma~\ref{lemImage d'un segment par une retraction 2}, $\pi_B:=\rho_B\circ \tau_\A$ is a $(\phi(x),\phi(y))$-path of $B$ from $x$ to $y$. Therefore $x\leq _B y$.  Let $\psi=\phi^{-1}:B\rightarrow \A$. Composing $\phi$ by some $w\in W^v$ if necessary, one can suppose that $\psi(y)-\psi(x)\in \overline{C^v_f}$. 

By Lemma~\ref{lemImage d'un segment par une retraction 2}, $\pi_\A$ is a $(\psi(x),\psi(y))$-path of $\A$. By Lemma~\ref{lemArrivee d'un Lambda chemin}, we deduce that $y-x\leq_{Q^\vee}\psi(y)-\psi(x)$. 

By Lemma~\ref{lemImage d'un lambda seg par un iso}, $\psi\circ \pi_B$ is an $(x,y)$-path of $\A$ from $\psi(x)$ to $\psi(y)$. By Lemma~\ref{lemArrivee d'un Lambda chemin}, we deduce that $\psi(y)-\psi(x)\leq_{Q^\vee} y-x$. Therefore $x-y=\psi(x)-\psi(y)$ and $\pi_\A$ is an $(x,y)$-path of $\A$.
\end{proof}

If $x,y\in \I$, one says that $x\leq y$ if there exists an apartment $A$ containing $x,y$ and such that $x\leq_A y$. By  Lemma~\ref{lemImage d'un segment ordonne par une retraction}, this does not depend on the choice of $A$: if $x\leq y$ then for all apartment $B$ containing $x,y$, one has $x\leq_B y$. However, one does not know yet that $\leq $ is a preorder: the proof of Th\'eor\`eme 5.9 of \cite{rousseau2011masures} uses (MAO).

\medskip

The following lemma is Lemma~3.6 of \cite{hebertGK}: 

\begin{lem}\label{lem3.6 de HebertGk}
Let $\tau:[0,1]\rightarrow \I$ be a segment such that $\tau(0)\leq \tau(1)$, such that $\tau(1)\in \A$ and such that there exists $\nu\in\overline{C^v_f}$ such that  $(\rho_{-\infty}\circ \tau)' =\nu$. Then $\tau([0,1])\subset \A$ and thus $\rho_{-\infty}\circ \tau =\tau$.
\end{lem}

\begin{proof}
Let $A$ be an apartment such that $\tau$ is a segment of $A$. Then $\tau$ is increasing for $\leq_A$ and thus $\tau$ is increasing for $\leq$. Let $x,y\in A$ be such that $\tau(t)=(1-t)x+ty$ for all $t\in [0,1]$. Let us first prove that $\tau$ is increasing for $\leq$. It suffices to prove that $x\leq y$. By (MA iii'), there exists $u\in ]0,1]$ such that there exists an apartment $A$ containing $\tau([0,u])$ and $-\infty$. Let $\phi:A\overset{-\infty}{\rightarrow} \A$. One has \[\phi(\tau(u))=\rho_{-\infty}(\tau(u))=\rho_{-\infty}(\tau(0))+u\nu=\phi(\tau(0))+u\nu,\] thus $\phi(\tau(u))\geq \phi(\tau(0))$ and hence $\tau(u)\geq \tau(0)$. As $\tau$ is a segment of $A$, it suffices to prove that there exists $u>0$ such that $\tau(u)\geq \tau(0)$.  Therefore $\tau$ is increasing for $\leq$.

Suppose that $\tau([0,1])\nsubseteq \A$. Let $u=\sup \{t\in [0,1]|\tau(t)\notin \A\}$. Let us prove that $\tau(u)\in \A$. If $u=1$, this is our hypothesis. Suppose $u<1$. Then by (MA2') applied to $\big]\tau(u),\tau(1)\big)$, $\A$ contains $\cl_A\big(]\tau(u),\tau(1))\big)$ and thus $\A$ contains $\tau(u)$.

By (MA iii'), there exists an apartment $B$ containing $\tau((0,u])\cup -\infty$ and by (MA4), there exists an isomorphism $\phi:B\overset{\tau(u)-\overline{C^v_f}}{\rightarrow} \A$. For all $t\in [0,u]$, near enough from $u$, one has $\phi(\tau(t))=\rho_{-\infty}(\tau(t))$. By hypothesis, for all $t\in [0,u]$, $\rho_{-\infty}(\tau(t)) \in \tau(u)-\overline{C^v_f}$. Therefore for $t$ near enough from $u$,  $\phi(\tau(t))=\tau(t)\in \A$: this is absurd by choice of $u$ and thus $\tau([0,1])\subset \A$.
\end{proof}

We can now prove Proposition~\ref{propSuppression de MAO}: $\I$ satisfies (MAO).

\begin{proof}
Let $x,y\in \A$ be  such that $x\leq_\A y$ and $B$ be an apartment containing $\{x,y\}$. We suppose that $y-x\in \overline{C^v_f}$. Let $\pi_\A:[0,1]\rightarrow \A$ mapping each $t\in [0,1]$ on $\rho_{-\infty}((1-t)x+_B ty)$. By Lemma~\ref{lemImage d'un segment ordonne par une retraction},  $\pi_\A$ is an $(x,y)$-path from $x$ to $y$. By Lemma~\ref{lemLambda chemin de x à x+ lambda}, $\pi_\A(t)=x+t(y-x)$ for all $t\in [0,1]$.  Then by Lemma~\ref{lem3.6 de HebertGk}, $\pi_\A(t)=(1-t)x+_B ty$ for all $t\in [0,1]$. In particular $[x,y]=[x,y]_B$ and thus $\I$ satisfies (MAO).
\end{proof}

\subsubsection{Equivalence of the axiomatics}

As each chimney or face contains an extended chimney or a local face of the same type, if $\cl\in \mathcal{CL}_{\Lambda'}$, (MA iii, $\cl$) implies (MA iii').  Therefore a masure of type $(2,\cl)$ is also a masure of type $3$. 

If $A$ is an apartment and $F$ is a filter of $A$, then $\cl_A(F)\subset \cl_A^\#(F)$. Therefore for all $\cl\in \mathcal{CL}_{\Lambda'}$,  (MA2', $\cl^\#$) implies (MA2', $\cl$) and (MA iii, $\cl^\#$) implies (MA iii, $\cl$).

\begin{lem}\label{lemType 1 implique Type 2}
Let $\cl\in \mathcal{CL}_{\Lambda'}$ and $\I$ be a masure of type $(1,\cl)$. Then $\I$ is a masure of type $(2,\cl)$.
\end{lem}

\begin{proof}
By Theorem~\ref{thmIntersection enclose fort}, $\I$ satisfies (MA ii). By cons\'equence 2.2 3) of \cite{rousseau2011masures}, $\I$ satisfies (MA iii, $\cl$).
\end{proof}

By abuse of notation if $\I$ is a masure of any type and if $\q$, $\q'$ are adjacent sectors of $\I$, we denote by $\q\cap \q'$ the maximal face of $\overline{\q}\cap \overline{\q'}$. This has a meaning by Section 3 of \cite{rousseau2011masures} for masures of type $1$ and by (MA ii) for masures of type $2$ and $3$.

\begin{lem}\label{lemType 3 donne (MA2 face)}
Let  $\I$ be a masure of type $3$. Let $A$ be an apartment. Let $\mathcal{X}$ be a filter of $A$ such that for all sector-germ $\s$ of $\I$, there exists an apartment containing $\mathcal{X}$ and $\s$. Then if $B$ is an apartment containing $\mathcal{X}$, $B$ contains $\cl^\#(\mathcal{X})$ and there exists an isomorphism $\phi:A\overset{\cl^\#(\mathcal{X})}{\rightarrow} B$.
\end{lem}

\begin{proof}
Let $\q$ and $\q'$ be sector-germs of $A$ and $B$ of the same sign. By (MA iii'), there exists an apartment $C$ containing $\q$ and $\q'$. Let $\q_1=\q,\ldots,\q_n=\q'$ be a gallery of sector-germs  from $\q$ to $\q'$ in $C$. One sets $A_1=A$ and $A_{n+1}=B$. By hypothesis, for all $i\in \llbracket 2,n\rrbracket$ there exists an apartment $A_i$ containing $\q_i$ and $\mathcal{X}$.  For all $i\in \llbracket 1,n-1\rrbracket$, $\q_i\cap \q_{i+1}$ is a splayed chimney and $A_i\cap A_{i+1}\supset \q_i\cap \q_{i+1}$. Therefore $A_i\cap A_{i+1}$ is enclosed and there exists $\phi_i:A_i\overset{A_i\cap A_{i+1}}{\rightarrow} A_{i+1}$. The set $A_n\cap A_{n+1}$ is also enclosed and there exists $\phi_n:A_n\overset{A_n\cap A_{n+1}}{\rightarrow} A_{n+1}$.

If $i\in \llbracket 1,n+1\rrbracket$, one sets $\psi_i=\phi_{i-1}\circ \ldots \circ \phi_1$. Then $\psi_i$ fixes $A_1\cap \ldots \cap A_i$.

Let $i\in \llbracket 1,n\rrbracket$ and suppose that $A_1\cap \ldots \cap A_{i}$ is enclosed in $A$. The isomorphism $\psi_i$ fixes $A_1\cap \ldots \cap A_{i}$ and thus we deduce that $A_1\cap \ldots \cap A_{i}=\psi_i(A_1\cap \ldots \cap A_{i})$ is enclosed in $A_i$. Moreover, $A_i\cap A_{i+1}$ is enclosed in $A_i$ and thus $A_1\cap\ldots\cap A_{i+1}$ is enclosed in $A_i$. Consequently $A_1\cap \ldots \cap A_{i+1}=\psi^{-1}_i(A_1\cap \ldots \cap A_{i+1})$ is enclosed in $A$. Let $X=A_1\cap \ldots \cap A_{n+1}$.  By induction, $X$ is enclosed in $A$ and $\phi:=\psi_n$ fixes $X$. As $X\supset \mathcal{X}$, we deduce that $X\in \cl^\#(\mathcal{X})$ and we get the lemma.

\end{proof}

\begin{lem}\label{lemÉquivalence des types 2 et 3}
Let $\I$ be a masure of type $3$. Then for all $\cl\in \mathcal{CL}_{\Lambda'}$, $\I$ satisfies (MA iii, $\cl$).
\end{lem}

\begin{proof}
Each face is contained in the finite enclosure of a local face and each chimney is  contained in the finite enclosure of an extended chimney. Thus by Lemma~\ref{lemType 3 donne (MA2 face)}, applied when $\mathcal{X}$ is a local face and a germ of a chimney, $\I$ satisfies (MA iii, $\cl^\#$). Consequently  for all $\cl\in \mathcal{CL}_{\Lambda'}$, $\I$ satisfies (MA iii, $\cl$), hence (MA3, $\cl$) and the lemma is proved.
\end{proof}

\begin{lem}\label{lemType3 implique (MA2')}
Let $\I$ be a masure of type $3$ and $\cl\in \mathcal{CL}_{\Lambda'}$. Then $\I$ satisfies (MA2', $\cl$).
\end{lem}

\begin{proof}
If $A$ is an apartment and $F$ is a filter of $A$, then $\cl(F)\subset \cl^\#(F)$. Therefore it suffices to prove that $\I$ satisfies (MA2', $\cl^\#$). We conclude the proof by applying Lemma~\ref{lemType 3 donne (MA2 face)} applied when $\mathcal{X}$ is a point, a germ of a preordered segment. 
\end{proof}

Using Proposition~\ref{propSuppression de MAO}, we deduce that  a masure of type $2$ or $3$ satisfies (MAO), as (MA4) is a consequence of (MA ii).

\begin{lem}\label{lemFixer une cheminee par son germe}
Let  $\I$ be a masure of type $3$. Let $\rr$ be a chimney of $\A$, $\rr=\rr(F^\ell,F^v)$, where $F^\ell$ (resp. $F^v$) is a local face (resp. vectorial face) of $A$. Let $\RR^\#=germ_\infty (\cl^\#(F^\ell,F^v))$. Let $A$ be an apartment containing $\rr$ and $\RR^\#$ and such that there exists $\phi:\A\overset{\RR^\#}{\rightarrow} A$. Then $\phi:\A\overset{\rr}{\rightarrow} A$.
\end{lem}

\begin{proof}
One can suppose that $F^v\subset \overline{C^v_f}$. Let $U\in \RR^\#$ such that $U$ is enclosed, $U\subset A\cap \A$ and such that $U$ is fixed by $\phi$. One writes $U=\bigcap_{i=1}^k D(\beta_i,k_i)$, with $\beta_1,\ldots,\beta_k\in \Phi$ and $(k_1,\ldots,k_r)\in \prod_{i=1}^r \Lambda_{\beta_i}'$.

 Let $\xi\in F^v$ be such that $U\in \cl (F^\ell+F^v+\xi)$. Let $J=\{i\in \llbracket 1,k\rrbracket|\ \beta_i(\xi)\neq 0\}$. For all $i\in \llbracket 1,r\rrbracket$,  one has $D(\beta_i,k_i)\supset n \xi$ for $n\gg 0$. Thus $\beta_i(\xi)>0$ for all $i\in J$. One has $U-\xi=\bigcap_{i=1}^k D(\beta_i,k_i+\beta_i(\xi))$. Let $\lambda\in [1,+\infty[$ be  such that for all $i\in J$, there exists $\tilde{k_i}\in \Lambda_{\beta_i}'$ such that  $k_i+\beta_i(\xi)\leq \tilde{k_i}\leq k_i+\lambda \beta_i(\xi)$.  Let $\tilde{U}=\bigcap_{i=1}^k D(\beta_i,\tilde{k_i})$. Then $ U-\xi\subset \tilde{U} \subset U-\lambda \xi $. Therefore, $\tilde{U}\in \rr$. Let $V'\in \rr$ be such that $V'\subset A\cap \A$ and such that $V'+F^v\subset V'$. Then $V:=\tilde{U}\cap V'\in \rr$. Let $v\in V$ and $\delta\subset F^v$ be the ray based at $0$ and containing $\xi$. By the proof of Proposition 5.4 of \cite{rousseau2011masures} (which uses only (MA1), (MA2'), (MA3), (MA4) and (MAO)), there exists $g_v:\A\overset{v+\delta}{\rightarrow} A$. As $V\subset U-\lambda\xi$, there exists a shortening $\delta'$ of $v+\delta$ contained in $U$.  Then $g_v^{-1}\circ \phi:\A\rightarrow \A$ fixes $\delta'$. Consequently, $g_v^{-1}\circ \phi$ fixes the support of $\delta'$ and thus $\phi$ fixes $v$: $\phi$ fixes $V$. Therefore $\phi$ fixes $\rr$ and the lemma follows.
\end{proof}

\begin{lem}\label{lemType 3 implique MA2''}
Let $\I$ be a masure  of type $3$ and $\cl\in \mathcal{CL}_{\Lambda'}$. Then  $\I$ satisfies (MA2'', $\cl$).
\end{lem}

\begin{proof}
Let $\rr=\cl(F^l,F^v)$ be a solid chimney of an apartment $A$ and $A'$ be an apartment containing $\rr$. One supposes that $A=\A$. Let $\rr^\#=\cl^\#(F^l,F^v)$ (resp. $\rr_e=F^l+F^v$) and  $\RR^\#$ (resp. $\RR_e$) be the germ of $\rr^\#$ (resp. $\rr_e$). By Lemma~\ref{lemType 3 donne (MA2 face)} applied with $\mathcal{X}=\mathfrak{R}_e$, there exists $\phi:A\overset{\mathfrak{R}^\#}{\rightarrow} A'$.   By Lemma~\ref{lemFixer une cheminee par son germe}, $\phi$ fixes $\rr$ and thus $\I$ satisfies (MA2'', $\cl$).

\end{proof}

We can now prove Theorem~\ref{thmAxiomatique des masures forte}: let $\cl\in \mathcal{CL}_{\Lambda'}$. By Lemma~\ref{lemType 1 implique Type 2}, a masure of type $(1,\cl)$ is also a masure of type $(2,\cl)$  and thus it is a masure of type $3$. By Lemma~\ref{lemÉquivalence des types 2 et 3}, Lemma~\ref{lemType3 implique (MA2')} and Lemma~\ref{lemType 3 implique MA2''}, a masure of type $3$  is  a masure of type $(1,\cl)$ which concludes the proof of the theorem.

\subsection{Friendly pairs in $\I$}

Let  $\A=(\A,W,\Lambda')$ be an apartment. Let $\I$ be a masure of type $\A$. We now use the finite enclosure $\cl=\cl^\#_{\Lambda'}$, which makes sense by Theorem~\ref{thmAxiomatique des masures}.  A family $(F_j)_{j\in J}$ of filters in $\I$ is said to be \textbf{friendly} if there exists an apartment containing $\bigcup_{j\in J} F_j$. In this section we obtain friendliness results for pairs of faces, improving results of Section 5 of \cite{rousseau2011masures}. We will use it to give a very simple axiomatic of masures in the affine case. These kinds of results  also  have an interest on their own: the definitions of the Iwahori-Hecke algebra of \cite{bardy2016iwahori} and of the  parahoric Hecke algebras of \cite{abdellatif2017completed} rely on the existence of apartments containing pairs of faces.

If $x\in \I$, $\epsilon\in \{-,+\}$ and $A$ is an apartment, one denotes by $\FF_x$ (resp. $\FF^\epsilon$, $\FF^\epsilon(A)$, $\CC_x$, $\ldots$) the set of faces of $\I$ based at $x$ (resp. and of sign $\epsilon$, and contained in $A$, the set of chambers of $\I$ based at $x$, $\ldots$). If $\X$ is a filter, one denotes by $\AC(\X)$ the set of apartments containing $\X$.

\begin{lem}\label{lemSeparation des chambres restreintes par des murs}
Let $A$ be an apartment of $\I$, $a\in A$ and $C_1,C_2\in \CC_a(A)$. Let $\mathcal{D}_a$ be the set of half-apartments of $A$  whose wall contains $a$. Suppose that $C_1\neq C_2$. Then there exists $D\in \mathcal{D}_a$ such that $D\supset C_1$ and $D\nsupseteq C_2$. 
\end{lem}

\begin{proof}
Let $C_1^v$ and $C_2^v$ be vectorial chambers of $A$ such that $C_1=F(a,C_1^v)$ and $C_2=F(a,C^v_2)$. Suppose that for all $D\in \mathcal{D}_a$ such that $D\supset C_1$, one has $D\supset C_2$. Let $X\in C_1$. There exists half-apartments $D_1,\ldots,D_k$ and $\Omega\in \mathcal{V}_A(a)$ such that $X\supset\bigcap_{i=1}^k D_i^\circ \supset \Omega\cap (a+C_1^v)$. 

Let $J=\{j\in \llbracket 1,k\rrbracket | \ D_j\notin \mathcal{D}_a\}$. For all $j\in J$, one chooses $\Omega_j\in \mathcal{V}_A(a)$ such that $D_j^\circ\supset \Omega_j$. If $j\in \llbracket 1,k\rrbracket \backslash J$, $D_j\supset C_1$, thus $D_j\supset C_2$ and hence $D_j^\circ \supset C_2$. Therefore, there exists $\Omega_j\in \mathcal{V}_A(a)$  such that $D_j^\circ\supset \Omega_j \cap (x+C^v_2)$. Hence \[X\supset \bigcap_{j=1}^k D_j^\circ\supset (\bigcap_{j=1}^k \Omega_j)\cap (x+C^v_2),\] thus $X\in C_2$ and $C_1\supset C_2$.

Let $D\in \mathcal{D}_a$ such that $D\supset C_2$. Suppose that $D\nsupseteq C_1$. Let $D'$ be the half-apartment opposite $D$. Then $D'\supset C_1$ and therefore $D'\supset C_2$: this is absurd. Therefore for all $D\in \mathcal{D}_a$ such that $D\supset C_2$, one has  $D\supset C_1$. By the same reasoning we just did, we deduce that $C_2\supset C_1$ and thus $C_1=C_2$. This is absurd and the lemma is proved.

\end{proof}

The following proposition improves Proposition 5.1 of \cite{rousseau2011masures}. It is the analogue of axiom (I1) of buildings (see the introduction).

\begin{prop}\label{propPaires amicales dans I}
Let $\{x,y\}$ be a friendly pair in $\I$. \begin{enumerate}

\item\label{itemPaire droite chambre} Let $A\in \AC(\{x,y\})$ and $\delta$ be a ray of $A$ based at $x$ and containing $y$ (if $y\neq x$, $\delta$ is unique) and $F_x\in \FF_x$. Then $(\delta,F_x)$ is friendly. Moreover, there exists  $A'\in \AC(\delta\cup F_x)$ such that there exists an isomorphism $\phi:A\overset{\delta}{\rightarrow}A'$.

\item\label{itemPaire de faces} Let $(F_x,F_y)\in \FF_x\times \FF_y$. Then $(F_x,F_y)$ is friendly.
\end{enumerate}

\end{prop}

\begin{proof}
We begin by proving~\ref{itemPaire droite chambre}. Let $C_x$ be a chamber of $\I$ containing $F_x$. Let $C$ be a chamber of $A$ based at $x$ and having the same sign as $C_x$. By Proposition~5.1 of \cite{rousseau2011masures}, there exists an apartment $B$ containing $C_x$ and $C$. Let $C_1=C,\ldots,C_n=C_x$ be a gallery in $B$ from $C$ to $C_x$. If $i\in \llbracket 1,n\rrbracket$, one sets $\mathcal{P}_i$: ``there exists an apartment $A_i$ containing $C_i$ and $\delta$ such that there exists an isomorphism $\phi:A\overset{\delta}{\rightarrow} A_i$''.
 The property $\mathcal{P}_1$ is true by taking $A_1=A$. Let $i\in \llbracket 1,n-1\rrbracket$ be such that $\mathcal{P}_i$ is true. 
 If $C_{i+1}=C_i$, then $\mathcal{P}_{i+1}$ is true. Suppose $C_i\neq C_{i+1}$. Let $A_i$ be an apartment containing $C_i$ and $\delta$. By Lemma~\ref{lemSeparation des chambres restreintes par des murs}, there exists a half-apartment $D$ of $A$ whose wall contains $x$ and such that $C_i\subset D$ and $C_{i+1}\nsubseteq D$. As $C_i$ and $C_{i+1}$ are adjacent, the wall $M$ of $D$ is the wall separating $C_i$ and $C_{i+1}$.  By (MA2), there exists an isomorphism $\phi:B\overset{C_i}{\rightarrow} A_i$. Let $M'=\phi(M)$ and $D_1$, $D_2$ be the half-apartments of $A_i$ delimited by $M'$. Let $j\in \{1,2\}$ such that $D_j\supset \delta$. By Proposition 2.9 1) of \cite{rousseau2011masures}, there exists an apartment $A_{i+1}$ containing $D_j$ and $C_{i+1}$. Let $\psi_i:A\overset{\delta}{\rightarrow} A_i$ and $\psi:A_i\overset{D_j}{\rightarrow} A_{i+1}$. Then $\psi\circ\psi_i:A\overset{\delta}{\rightarrow}A_{i+1}$. Therefore $\mathcal{P}_{i+1}$ is true. Consequently, $\mathcal{P}_n$ is true, which proves~\ref{itemPaire droite chambre}.

Let us prove~\ref{itemPaire de faces}, which is very similar to~\ref{itemPaire droite chambre}. As a particular case of~\ref{itemPaire droite chambre}, there exists an apartment $A'$ containing $F_x$ and $y$. Let $C_y$ be a chamber of $\I$ containing $F_y$. Let $C$ be a chamber of $A'$ based at $y$ and of the same sign as $F_y$. Let $C_1=C,\ldots,C_n=C_y$ be a gallery of chambers from $C$ to $C_y$ (which exists by Proposition 5.1 of \cite{rousseau2011masures}). By the same reasoning as above, for all $i\in \llbracket 1,n \rrbracket$, there exists an apartment containing $F_x$ and $C_i$, which proves~\ref{itemPaire de faces}.
\end{proof}

\subsection{Axiomatic of masures in the affine case}

In this section, we study the particular case of masures associated to irreducible affine Kac-Moody matrix $A$, which means that $A$ satisfies condition (aff) of Theorem $4.3$ of \cite{kac1994infinite}.

Let $\mathcal{S}$ be a generating root system associated to an irreducible and affine Kac-Moody matrix and $\A=(\mathcal{S},W,\Lambda')$ be an apartment. By Section 1.3 of  \cite{rousseau2011masures}, one has $\mathring{\T}=\{v\in \A|\delta(v)>0\}$ for some imaginary root $\delta\in Q^+\backslash\{0\}$ and $\T=\mathring{\T}\cup \A_{in}$, where $\A_{in}=\bigcap_{i\in I}\ker(\alpha_i)$.

We fix an apartment $\A$ of affine type.

\medskip 

Let (MA af i)=(MA1). 

Let (MA af ii) : let $A$ and $B$ be two apartments. Then $A\cap B$ is enclosed and there exists $\phi:A\overset{A\cap B}{\rightarrow} B$.

Let (MA af iii)= (MA iii).

\medskip

The aim of this subsection is to prove the following theorem: 

\begin{thrm}\label{thmAxiomatique dans le cas affine}
 Let $\I$ be a construction of type $\A$ and $\cl\in \mathcal{CL}_{\Lambda'}$. Then $\I$ is a masure  for $\cl$ if and only if $\I$ satisfies (MA af i), (MA af ii) and (MA af iii, $\cl$) if and only if $\I$ satisfies (MA af i), (MA af ii) and (MA af iii, $\cl^\#$).
\end{thrm}

\begin{remark}
Actually, we do not know if this axiomatic is true for masures associated to indefinite Kac-Moody groups. We do not know if the intersection of two apartments is always convex in a masure.
\end{remark}

The fact that we can exchange (MA af iii, $\cl^\#$) and (MA af iii, $\cl$) for all $\cl\in \mathcal{CL}_{\Lambda'}$ follows from Theorem~\ref{thmAxiomatique des masures forte}. The fact that a construction satisfying (MA af ii) and (MA af iii, $\cl^\#$) is a masure is clear and does not use the fact that $\A$ is associated to an affine Kac-Moody matrix.  It remains to prove that a masure of type $\A$ satisfies (MA af ii), which is the aim of this subsection. 

\begin{lem}\label{lemConvexite des intersections d'apparts cas affine ordonne}
Let  $A$ and $B$ be two apartments such that there exist $x,y\in A\cap B$ such that $x\mathring{\leq }y$ and $x\neq y$. Then $A\cap B$ is convex.
\end{lem}

\begin{proof}
One identifies $A$ and $\A$. Let $a,b\in \A\cap B$. If $\delta(a)\neq \delta(b)$, then $a\leq b$ or $b\leq a$ and $[a,b]\subset B$ by (MAO). Suppose $\delta(a)=\delta(b)$. As $\delta(x)\neq \delta(y)$, one can suppose that $\delta(a)\neq \delta(x)$. Then $[a,x]\subset B$. Let $(a_n)\in [a,x]^\N$ be  such that $\delta(a_n)\neq \delta (a)$ for all $n\in \N$ and $a_n\rightarrow a$. Let $t\in [0,1]$. Then $t a_n+(1-t)b \in B$ for all $n\in \N$ and by Proposition~\ref{propIntersection fermee}, $ta+(1-t)b\in B$: $\A\cap B$ is convex.
\end{proof}

\begin{lem}\label{lemConvexite dans le cas affine}
Let $A$ and $A'$ be two apartments of $\I$. Then $A\cap A'$ is convex. Moreover, if $x,y\in A\cap A'$, there exists an isomorphism $\phi:A\overset{[x,y]_A}{\rightarrow} A'$.
\end{lem}

\begin{proof}
Let $x,y\in A\cap A'$ be such that $x\neq y$. Let $C_x$ be a chamber of $A$ based at $x$ and $C_y$ be a chamber of $A'$ based at $y$. Let $B$ be an apartment containing $C_x$ and $C_y$, which exists by Proposition~\ref{propPaires amicales dans I}. By Lemma~\ref{lemConvexite des intersections d'apparts cas affine ordonne}, $A\cap B$ and $ A'\cap B$ are convex and by Proposition~\ref{propIsomorphisms fixant une partie convexe}, there exist isomorphisms $\psi:A\overset{ A\cap B}{\rightarrow} B$ and $\psi':B\overset{A'\cap B}{\rightarrow} A'$. Therefore $[x,y]_A=[x,y]_B=[x,y]_{A'}$. Moreover, $\phi=\psi'\circ \psi$ fixes $[x,y]_A$ and the lemma is proved.
\end{proof}

\begin{thrm}\label{thmConvexite cas affine}
Let $A$ and $B$ be two apartments. Then $A\cap B$ is enclosed and there exists an isomorphism $\phi:A\overset{A\cap B}{\rightarrow} B$.
\end{thrm}

\begin{proof}
The fact that $A\cap B$ is enclosed is a consequence of Lemma~\ref{lemConvexite dans le cas affine} and Proposition~\ref{propIntersection enclose quand convexe}. By Proposition~\ref{propÉcriture de l'intersection comme union des Pi, cas general}, there exist $\ell\in \N$, enclosed subsets $P_1,\ldots,P_\ell$ of $A$ such that $\supp (A\cap B)=\supp(P_j)$ and isomorphisms $\phi_j:A\overset{P_j}{\rightarrow}B$ for all $j\in \llbracket 1,\ell\rrbracket$. Let $x\in \In_r (P_1)$ and $y\in A\cap B$. By Lemma~\ref{lemConvexite dans le cas affine}, there exists $\phi_y:A\overset{[x,y]}{\rightarrow}B$. Then $\phi_y^{-1}\circ \phi_1$ fixes a neighborhood of $x$ in $[x,y]$ and thus $\phi_1$ fixes $y$, which proves the theorem. \end{proof}

\end{document}